\documentclass[a4paper, 11pt]{article}

\usepackage[DIV=12]{typearea}

\usepackage[T1]{fontenc}
\usepackage[utf8]{inputenc}

\usepackage{amsmath}
\usepackage{amsfonts}
\usepackage{amssymb}
\usepackage{amsthm}
\usepackage{mathtools}
\usepackage{xfrac}

\usepackage{braket}
\usepackage{enumerate} 
\usepackage{cancel}
\usepackage{dsfont} 

\usepackage{hyphsubst}
\usepackage{csquotes}

\usepackage[sc]{mathpazo}
\usepackage{bigints}
\usepackage{xcolor}

\usepackage[backend=biber, 
            language=english,%
            style=alphabetic,%
            giveninits=true,
            maxbibnames=5, 
            doi=false,
            isbn=false,
            url=false
           ]{biblatex}

\AtEveryBibitem{%
  \clearfield{pages}%
  \clearfield{month}%
}

\usepackage{hyperref}

\hypersetup{%
    colorlinks=false, hidelinks, bookmarksnumbered,pdfencoding=auto,
    bookmarksopen=true, bookmarksopenlevel=1, unicode,
    pdftitle={Dynamic Inverse Wave Problems -- Part II:\@ Operator Identification and Applications}, 
    pdfsubject={Dynamic Inverse Wave Problems -- Part II:\@ Operator Identification and Applications}, 
    pdfauthor={Thies Gerken},%
    pdfkeywords={Inverse Problems, Hyperbolic, Evolution Equations, Time-dependent, PDE}
}

\usepackage{cleveref}

\newtheorem{theorem}{Theorem}[section]
\newtheorem{corollary}[theorem]{Corollary}
\newtheorem{lemma}[theorem]{Lemma}

\theoremstyle{definition}

\newtheorem{example}[theorem]{Example}

\theoremstyle{remark}

\numberwithin{equation}{section}

\DeclareMathOperator{\spt}{supp}

\DeclareMathOperator{\divv}{div}
\DeclareMathOperator{\curl}{curl}

\newcommand\Q{\ensuremath{\mathbb Q}}
\newcommand\R{\ensuremath{\mathbb R}}
\newcommand\N{\ensuremath{\mathbb N}}

\newcommand\map[3]{\ensuremath{{#1}\,\colon \,{#2}\to{#3}}}

\newcommand{\subnorm}[2]{{{\left \|#2\right \|}_{#1}}}
\newcommand{\norm}[1]{{\subnorm{{}}{#1}}}
\newcommand{\tnorm}[1]{{\|#1 \|}}
\newcommand{\norms}[1]{\|#1\|}

\newcommand\dt{\ensuremath{{\,\mathrm{d}t}}}

\newcommand\dx{\ensuremath{{\,\mathrm{d}x}}}

\newcommand\dd{\ensuremath{{\,\mathrm{d}}}}
\newcommand\ddt{\ensuremath{{\frac{\dd}{\dt}\,}}}
\newcommand\scp[2]{{\left({#1},{#2} \right)}}
\newcommand\dup[2]{{\langle{#1},{#2} \rangle}}

\newcommand{\assign}{=}

\newcommand\Lsa{\mathcal{L}^{\mathrm{sa}}}

\renewcommand{\phi}{\varphi}
\renewcommand{\epsilon}{\varepsilon}
\newcommand{\eps}{\epsilon}

\title{Dynamic Inverse Wave Problems -- Part II:\\ Operator Identification and Applications} 

\author{Thies Gerken\thanks{Center for Industrial Mathematics, Universit\"at Bremen, Germany; \texttt{tgerken@math.uni-bremen.de}}}

\date{\today}

\begin{document}

\maketitle

\begin{abstract}
  We present a framework which enables the analysis of dynamic inverse problems for wave phenomena that are modeled through second-order hyperbolic PDEs. 
  This includes well-posedness and regularity results for the forward operator in an abstract setting, where the operators in an evolution equation represent the unknowns. 
  We also prove Fréchet-differentiability and local ill-posedness for this problem.
  We then demonstrate how to apply this theory to actual problems by two example equations motivated by linear elasticity and electrodynamics. 
  For these problems it is even possible to obtain a simple characterization of the adjoint of the Fréchet-derivative of the forward operator, which is of particular interest for the application of regularization schemes. 
\end{abstract}


\section{Introduction}
 
Our main motivation behind this work is the identification of time-dependent quantities that govern wave propagation.
The first example for such a setting is the reconstruction of the wave speed or mass density in a wave equation from measurements of the wave field.
We thereby continue the work done in~\cite{gerken_reconstruction_2017}, where only a zero-order potential was sought. 
To be more precise, the equation under consideration in this situation is 
\begin{equation*}
  \frac{1}{\rho(t,x)} \left(\frac{u'(t,x)}{c(t,x)^2}\right)' - \divv \left( \frac{\nabla u(t,x)}{\rho(t,x)} \right) = f, 
\end{equation*}
together with suitable initial- and boundary conditions. In this setting the right-hand side $f$ is known, and either $c$ or $\rho$ is to be identified. The corresponding problem with static parameters was previously analyzed in~\cite{kirsch_linearization_2014} and~\cite{kirsch_seismic_2014}.

Another scenario of interest can be found in elasticity. Here one can discuss the problem of reconstructing time-dependent Lamé parameters, thus lifting~\cite{lechleiter_identifying_2017} into the world of dynamic inverse problems. 
Both the classic- and the elastic wave equation already share similar theory for existence, uniqueness and regularity of the solution because they can both be written as evolution equations. 
Inspired by~\cite{kirsch_inverse_2016} and~\cite{blazek_mathematical_2013} we also developed a common approach to the analysis of the inverse problems, but based on a second-order formulation and with the strong emphasis on time-dependent parameters. 

The general formulation which we consider throughout this article is the evolution equation 
\begin{equation}\label{eq:intro:abstract}
  \ddt C(t)u'(t) + B(t)u'(t) + A(t)u(t) + Q(t)u(t) = f(t),
\end{equation}
to be solved on a finite time interval $I=(0,T)$ with $T>0$. The unknowns are the linear operators $A, B, C$ and $Q$, and the forward operator of the problem maps them onto the solution $u$ of this equation. In applications one would subsequently apply a measurement operator, which restricts the knowledge of $u$, e.g.\ to boundary data. As long as this operator is linear or at least Fréchet-differentiable, this would not impact the analysis done here. 

We give a short motivation why we chose the operators in~\eqref{eq:intro:abstract}. The operators $A$ and $C$ contain the second-order differential operators in space and time, respectively, and are therefore of particular interest. 
We include another operator $Q$, which can only act on lower spatial derivatives of $u$, like a potential in the wave equation (as in~\cite{gerken_reconstruction_2017}). 
Such an operator might arise from the linearization of a previously semi-linear equation. 
By not combining it with $A$ we can achieve lower regularity assumptions on this part of the equation. 
The operator $B$ is not only valueable to introduce damping into the wave propagation, but also gives more flexibility in the positioning of an unknown parameter between the time derivatives in the highest order term. 
Without it, handling of $u''/\rho$ like in the introductory wave equation would not be possible. 
By including $B$, we can re-write this as $(u'/\rho)' - u' \rho' / \rho^2$. Using the regularity results of~\cite{gerken_dynamic_2018} one can conclude that these two formulations are equivalent.

This article is organized as follows. Section~\ref{section:abstract} contains the abstract framework based on equation~\eqref{eq:intro:abstract}. After establishing a well-defined forward operator defined on an open subset of a Banach space we can analyse its differentiability in Section~\ref{section:abstract:frechet} and try to understand the adjoint of the resulting Fréchet derivative, which is the focus of Section~\ref{section:abstract:adjoint}. We close the abstract theory by showing local ill-posedness of the problem and also ill-posedness of its linearizations. In Sections~\ref{section:elasticity} and~\ref{section:maxwell} we then demonstrate how easy it is to apply this abstract theory to actual PDEs using the elastic wave equation and a model for electrodynamics based on Maxwell's equations as examples. 

\section{Abstract Inversion}%
\label{section:abstract}

Let $V, H$ be separable Hilbert spaces, with a compact and dense embedding $V\hookrightarrow H$. 
Without loss of generality we assume $\norm{\cdot}_H \leq \norm{\cdot}_V$. 
By identifying $H$ with $H^*$, but not doing so with $V$, we obtain a Gelfand triple $V\subset H \subset V^*$. 

First we would like to make a few remarks on our notation. With $W^{k,p}(I; X)$ we denote the usual Bochner space of functions that take values in the Banach space $X$. For their definition we refer to~\cite{zeidler_nonlinear_1985}. If it is not indicated otherwise, then $\scp{\cdot}{\cdot}$ and $\dup{\cdot}{\cdot}$ denote the inner product of $H$ and the dual product of $V^*$ and $V$, respectively. 
Further, we write $\mathcal L(X, Y)$ for the space of linear and continuous operators between normed spaces $X$ and $Y$, with the shorthand notation $\mathcal L(X)$ if $X=Y$. 
For operators belonging to $L^\infty(I; \mathcal L(X,Y))$ we denote their \emph{realization} using calligraphic font, i.e.\ for some $F\in L^\infty(I; \mathcal L(X,Y))$ the operator $\map{\mathcal F}{L^2(I; X)}{L^2(I; Y)}$ is defined by 
\begin{equation*}
(\mathcal F v)(t) = F(t)v(t),
\end{equation*}
which is valid for almost all $t\in I$ if $v\in L^2(I; X)$.

In the remainder of this section we analyze the operator $S$, which maps the operators $A, B, C, Q$ to the solution $u\in L^2(I; V)\cap H^1(I; H)$ of the problem 
\begin{subequations}\label{eq:abstract:problem}
  \begin{gather}
  (\mathcal Cu')'+\mathcal Bu' + (\mathcal A + \mathcal Q) u = f \text{ in } L^2(I; V^*),\\
    u(0)=u_0\text{ in } H,\ (\mathcal C u')(0)=u_1\text{ in } V^*.
  \end{gather}
\end{subequations}
Each of the operators may be time-dependent, and to make the above equations well-defined we require $A\in L^\infty(I; \mathcal L(V, V^*))$, $B\in L^\infty(I; \mathcal L(H))$, $C\in L^{\infty}(I; \mathcal L(H))$ und $Q\in L^\infty(I; \mathcal L(V, H))$.
To ensure that this equation is of hyperbolic type we have to assume $A(t)$ and $B(t)$ to be self-adjoint and coercive, i.e. $\scp{C(t)\phi}{\phi}\geq c_0 \norm{\phi}^2_H$ and $\dup{A(t)\psi}{\psi} \geq a_0 \norm{\psi}^2_V$ for all $\phi\in H, \psi \in V$ and almost all $t\in I$ with constants $a_0, c_0 > 0$.
We note that the case where $A$ only fulfills the weaker Gårding-inequality $\dup{A(t)\psi}{\psi} \geq a_0 \norm{\psi}^2_V - \lambda \norm{\psi}^2_H$ with $\lambda\in\R$ can be remedied by replacing $A$ with $A+\lambda I$ and $Q$ with $Q-\lambda I$.

For the definition of the solution operator to~\eqref{eq:abstract:problem} we need function spaces that capture these restrictions on $A$ and $C$. 
Therefore we define for Hilbert spaces $Z$ the set
\begin{equation*}
  \Lsa(Z, Z^*) \assign \Set{G\in \mathcal L(Z, Z^*) | G^* = G}.
\end{equation*}
Here we identify $Z^{**}$ with $Z$, i.e. $G, G^*\in \mathcal L(Z, Z^*)$. 
Because we also identify $H$ with $H^*$ this also gives rise to $\Lsa(H)$. 
In this way we obtain a closed subspace of $\mathcal L(Z, Z^*)$, i.e. $\Lsa(Z, Z^*)$ is a Banach space when it is equipped with the operator norm. 
The \enquote{natural} set of permissible $A(t)$ and $C(t)$ can then be expressed through the notation 
\begin{equation*}
  \Lsa_\alpha(Z, Z^*) \assign \Set{G\in \Lsa(Z, Z^*) | \dup{Gz}z \geq \alpha \norm{z}^2_Z \text{\ for all $z\in Z$}},
\end{equation*}
where $\alpha$ is a positive constant. 
Conditions for the existence and uniqueness of $u$ then read as follows.

\begin{lemma}\label{lemma:abstract:wellposed}
  Let $A\in W^{1,\infty}(I; \Lsa(V,V^*))$ with $A(t)\in \Lsa_{a_0}(V, V^*)$ and $C\in W^{1,\infty}(I; \Lsa(H))$ with $C(t)\in \Lsa_{c_0}(H)$ for almost all $t\in I$ for some $a_0, c_0>0$. 
  Furthermore assume that $B\in W^{1,\infty}(I; \mathcal L(H))$, $Q\in W^{1,\infty}(I; \mathcal L(V, H))$, $f\in L^2(I; H) \cup H^1(I; V^*)$, $u_0\in V$, and $u_1\in H$. 
  Then there exists a uniquely determined $u\in L^2(I; V)\cap H^1(I; H)$ with $(\mathcal Cu')' \in L^2(I; V^*)$ solving~\eqref{eq:abstract:problem}.
  Furthermore the solution $u$ continuously depends on the data $u_0$, $u_1$ and $f$ as well as on the operators $A$, $B$, $C$ and $Q$, using the natural norms in the spaces above.
\end{lemma}

\begin{proof}
  See e.g.~\cite{lions_non-homogeneous_1972} or~\cite{zeidler_nonlinear_1985}.
\end{proof}

If $Q$ does not represent a first-order differential operator, i.e. $Q(t)\in \mathcal L(H)$, then the differentiability assumption on $Q$ in the theorem can be dropped and therefore $Q\in L^\infty(I; \mathcal L(H))$ would suffice.

For a proper analysis of the differentiability of $S$ we need the operator to be defined on an open subset of a Banach space. 
Unfortunately, the sets of all $A$ and $C$ that satisfy the coercivity constraint are not open because $\Lsa_\alpha(Z, Z^*)$ is closed.
The interior of this set is not obtained by simply using \enquote{$>$} instead of \enquote{$\geq$} in its definition (and restrict the condition to  $z\neq 0$) because this set is also not open. 
We would like to demonstrate this by a simple example. 

\begin{example}
  Let $\alpha\in [0,1)$ and $Z \assign L^2([\alpha, 1]; \R)$. 
  We define $F\in \Lsa(Z)$ for $v\in Z$ through
  \begin{equation*}
    F(v) = (x\mapsto x v(x)) \in Z,
  \end{equation*}
  i.e. $\scp{F(v)}v = \int_\alpha^1 x v(x)^2 \dx > \alpha \norm{v}_Z^2$ for all $v\neq 0$, in particular $F\in \Lsa_\alpha(Z)$. 
  We also define the family of operators
  \begin{equation*}
    F_\epsilon(v) \assign (x\mapsto (x-\epsilon) v(x)).
  \end{equation*}
  with $\eps>0$. 
  We denote by $1_\Omega$ the characteristic function of $\Omega$, see that 
  \begin{equation*}
    \scp{F_\eps(1_{[\alpha, \alpha+\eps]})}{1_{[\alpha, \alpha+\eps]}} = \int_\alpha^{\alpha+\eps} x-\eps \dx = (\alpha - \eps/2) \eps < \eps = \norm{1_{[\alpha, \alpha+\eps]}}^2_Z
  \end{equation*}
  and conclude $F_\epsilon \notin \Lsa_\alpha(Z)$. 
  But $\norm{F-F_\epsilon}_{\mathcal L(Z)} \leq \epsilon$, so $F_\epsilon \to F$ when $\eps\to 0$. 
  Therefore $F$ does not belong to the interior of $\Lsa_\alpha(Z)$.
\end{example}

Nevertheless, the interior of $\Lsa_\alpha(Z, Z^*)$ is not empty, and we can give a short formula for it.

\begin{lemma}\label{lemma:abstract:lsa_interior}
  The interior of $\Lsa_\alpha(Z, Z^*)$ is given by
  \begin{equation*}
    {\Lsa_\alpha(Z, Z^*)}^\circ = \bigcup_{\epsilon > 0} \Lsa_{\alpha+\epsilon}(Z, Z^*).
  \end{equation*}
\end{lemma}

\begin{proof}
  We set $M=\bigcup_{\epsilon > 0} \Lsa_{\alpha+\epsilon}(Z, Z^*)$ and show that it is the biggest open subset of $\Lsa_\alpha(Z, Z^*)$. 
  It is obvious that $M\subset \Lsa_\alpha(Z, Z^*)$. 
  We continue by proving that $M$ is open. 
  Let $G\in M$, which means there is $\epsilon_0>0$ such that $G\in \Lsa_{\alpha+\epsilon_0}(Z, Z^*)$. 
  For every $F\in B(G, \sfrac{\epsilon_0}2)$ (ball around $G$ with respect to the operator norm) and $v\in Z$ we have
  \begin{equation*}
    \dup{Fv}v = \dup{Gv}v + \dup{(F-G)v}v \geq (\alpha+\epsilon_0) \norm{v}^2 - \norm{F-G} \norm{v}^2 \geq (\alpha+\sfrac{\epsilon_0}2) \norm{v}^2,
  \end{equation*}
  which means that $F\in \Lsa_{\alpha+\sfrac{\epsilon_0}2}(Z, Z^*) \subset M$.
  As a last step we show that every $G \in \Lsa_\alpha(Z, Z^*) \setminus M$ can be approximated by operators that belong to $\Lsa(Z, Z^*)\setminus \Lsa_\alpha(Z, Z^*)$.
  Since $G\notin M$ there exists a sequence $(v_k)_{k\in\N}\subset Z$ with $\dup{Gv_k}{v_k} = (\alpha+\sfrac 1{(2k)})\norm{v_k}^2$.
  We set $G_k = G - \sfrac1k \, I_{Z\to Z^*}$, where $I_{Z\to Z^*}$ denotes the canonical embedding of the hilbert space $Z$ in its dual space. 
  It is easy to verify that $G_k \to G$ for $k\to\infty$ as well as $\dup{G_k v_k}{v_k} = (\alpha - \sfrac 1{(2k)}) \norm{v_k}^2$, so none of the $G_k$ belongs to $\Lsa_{\alpha}(Z, Z^*)$.
\end{proof}

We conclude that $\map S{D(S)\subset X}Y$ is well-defined when we fix $f, u_0$ and $u_1$ and make the definitions
\begin{align}\label{eq:abstract:space_x}
  X &= W^{1,\infty}(I; \Lsa(V,V^*)) \times W^{1,\infty}(I; \mathcal L(H)) \times W^{1,\infty}(I; \Lsa(H)) \times W^{1,\infty}(I; \mathcal L(V,H)), \\
  D(S) &= \Big \{(A, B, C, Q)\in X\ |\ A(t)\in \Lsa_{a_0+\epsilon}(V, V^*) \text{ and } C(t)\in\Lsa_{c_0+\epsilon}(H) \\
  &\hspace{10em} \text{ for almost all $t\in I$ for some $\epsilon > 0$} \Big \} \text{, and} \notag \\
  Y &= L^\infty(I; V)\cap W^{1,\infty}(I; H).
\end{align}
Again, in the case $Q(t)\in \mathcal L(H)$ we could omit the differentiability assumption on $Q$ in the definition of $X$.
The operator $S$ as given above is defined on an open subset of a Banach space and maps into another Banach space.
However, we will see that for this $S$ we are not able to show Fréchet-differentiability in $A$ or $C$. 
As we will see, this is due to a lack of regularity in $u=S(A, B, C, Q)$. 
Hence, we state the regularity result from~\cite{gerken_dynamic_2018} that will provide the required smoothness. 

For $k\geq 0$ let $u_{k+2}$ be given via
\begin{align*}
  C(0)u_{k+2} &= f^{(k)}(0) - ((k+1) C'(0) + B(0)) u_{k+1}\\
  &\ \ - \sum_{j=0}^{k} \left[ {k \choose j} (A^{(j)}(0) + Q^{(j)}(0)) + {k \choose j+1} B^{(j+1)}(0) +  {k+1 \choose j+2} C^{(j+2)}(0)\right] u_{k-j}. \notag
\end{align*}
Due to $C\in W^{1,\infty}(I; \Lsa(H))$ we know that $C$ is continuous, and from its coercivity we conclude that $C(0)$ is invertible. 
Therefore $u_{k+2}$ is well-defined as long as the right-hand side of the above equation is an element of $H$. 
With this notation we get the following result. 

\begin{theorem}\label{theorem:abstract:regularity}
    Let $k\in\N$ and suppose that $A\in W^{k+1,\infty}(I; \Lsa(V,V^*))$ with $A(t)\in \Lsa_{a_0}(V, V^*)$ and $C\in W^{k+1,\infty}(I; \Lsa(H))$ with $C(t)\in \Lsa_{c_0}(H)$ for almost all $t\in I$ and for some $a_0, c_0>0$. 
    Furthermore let $Q\in W^{k,\infty}(I; \mathcal L(V,H))$, $B\in W^{k,\infty}(I; \mathcal L(H))$,
    $f\in H^k(I; H)\cup H^{k+1}(I; V^*)$, $u_j\in V$ ($j=0,\dots, k$) and $u_{k+1}\in H$ be fulfilled.
    Then the unique solution $u$ of problem~\eqref{eq:abstract:problem} lies in $H^k(I; V)\cap H^{k+1}(I; H)$ with $(\mathcal C u^{(k+1)})' \in L^2(I; V^*)$ and satisfies the energy estimate
    \begin{equation*}
      \norm{u}_{W^{k,\infty}(I; V)}^2 + \tnorm{u^{(k+1)}}_{L^{\infty}(I; H)}^2 \leq \Lambda \left(\sum_{j=0}^k \norm{u_j}_V^2 + \norm{u_{k+1}}_H^2 + \norm{f}^2 \right)
    \end{equation*}
    where $f$ is measured in either the $H^k(I; H)$- or the $H^{k+1}(I; V^*)$ norm and $\Lambda=\Lambda(k)$ is a constant depending continuously on $1/{c_0}$, $1/{a_0}$, $T$ and the operators $A, B, C, Q$, measured in the spaces above.
\end{theorem}

The compatibility conditions $u_j\in V$ for $j=0,\dots, k$ and $u_{k+1}\in H$ read 
\begin{equation*}
  u_0,\, u_1 \in V,\ C(0) u_2 = f(0) - [C'(0)+B(0)]u_1 - [A(0)+Q(0)] u_0 \in H 
\end{equation*}
in the case $k=1$.
%
%
They encode \enquote{spatial} regularity of the operators and its time derivatives at the initial time and are (in general) nonlinear in the tupel $(A, B, C, Q)$. 
They can be linearized by making suitable additional assumptions on the operators and the data for the evolution equation, thus enabling them to be incorporated in the Banach space $X$. 
In this article we opt for the simplest solution by requiring homogeneous initial values $u_0=u_1=0$, $f^{(j)}(0) =0$ for $j=0, \dots, k-2$ and $f^{(k-1)}(0)\in H$, thereby avoiding any additional constraints in the space $X$. 
This allows for an easier notation, but any linear constraints that enforce the compatibility conditions would result in a similar analysis. 
Under these assumptions on $f$ and vanishing initial values we can also view $S$ for all $k\geq 1$ as the operator $\map S{D(S)\cap X^{(k)}\subset X^{(k)}}{Y^{(k)}}$ with
\begin{align}
  X^{(k)} &= W^{k+1,\infty}(I; \Lsa(V,V^*)) \times W^{k,\infty}(I; \mathcal L(H)) \times W^{k+1,\infty}(I; \Lsa(H)) \times W^{k,\infty}(I; \mathcal L(V,H)),\label{eq:abstract:space_xk}\\
  Y^{(k)} &= W^{k,\infty}(I; V)\cap W^{k+1,\infty}(I; H).
\end{align}
We extend this to $k=0$ by setting $X^{(0)}=X$ and $Y^{(0)}=Y$. 
To avoid having to repeat the conditions that $f$ has to fulfill in every assertion we define the set 
\begin{align*}
  \mathcal F^{(k)} = \Big \{ f\in H^k(I; H)\cup H^{k+1}(I; V^*) \ \Big|\ &\text{$f^{(k-1)}(0)\in H$ if $k\geq 1$ and} \\
  &\text{$f^{(j)}(0) = 0$ for all $j=0, \dots, k-2$ if $k\geq 2$} \Big \}
\end{align*}
of admissible right-hand sides.


\subsection{Fréchet-Differentiability}%
\label{section:abstract:frechet}

A formal application of the product rule shows that e.g. $u_h = \partial_A S(p)[h]$ should solve the same (linear) equation as $u$, but with the right-hand side $t\mapsto -h(t)[u(t)]$, where $u=S(p)$ is the solution of the direct problem. 
The same argument can be made for the other operators. 
Therefore we make the hypothesis that for each \enquote{symbol} $x\in \{A, B, C, Q\}$ the derivative $u_h = \partial_x S(p)[h]$ solves for each $p=(A, B, C, Q)\in D(S)$ the equation 
\begin{equation}\label{eq:abstract:linearized_general}
  (\mathcal Cu_h')'+\mathcal Bu_h' + (\mathcal A + \mathcal Q) u_h = g_x(u)[h]
\end{equation}
in $L^2(I; V^*)$ and possesses homogeneous initial values. 
The form of the right-hand side depends on the direction of the derivative and is given by  
\begin{alignat*}{3}
  g_A(v)[H] &= -\mathcal H[v] = - H(\cdot)[v(\cdot)], \qquad && g_C(v)[H] = -(\mathcal H[v'])', \\
  g_B(v)[H] &= -\mathcal H[v'], \qquad && g_Q(v)[H] = -\mathcal H[v].
\end{alignat*}
The right-hand sides for $A$ and $Q$ are the same, but $g_A$ and $g_Q$ will map between different spaces. 
We have to ensure that $g_x$ maps either into $L^2(I; H)$ or $H^1(I; V^*)$ in order to use \Cref{lemma:abstract:wellposed} to conclude that a unique solution $u_h\in Y^{(0)}$ of~\eqref{eq:abstract:linearized_general} exists. 
The natural choice of domains and ranges for the $g_x$ that facilitate this are 
\begin{align*}
  &\map {g_A(\cdot)[\cdot]}{H^1(I; V)\times W^{1,\infty}(I; \mathcal L(V, V^*))}{H^1(I; V^*)}, \\
  &\map {g_B(\cdot)[\cdot]}{H^1(I; H)\times L^\infty(I; \mathcal L(H))}{L^2(I; H)}, \\
  &\map {g_C(\cdot)[\cdot]}{H^2(I; H)\times W^{1,\infty}(I; \mathcal L(H))}{L^2(I; H)}, \\
  &\map {g_Q(\cdot)[\cdot]}{L^2(I; V)\times L^\infty(I; \mathcal L(V, H))}{L^2(I; H)}.
\end{align*}
This way we obtain continuous bilinear forms, e.g. 
\begin{equation*}
    g_A(\cdot)[\cdot] \in \mathcal L(H^1(I; V),\ \mathcal L(W^{1,\infty}(I; \mathcal L(V, V^*)),\ H^1(I; V^*)))
\end{equation*}
and can already deduce that $u\in Y^{(0)}$ is not enough to apply $g_A$ or $g_C$ to it. 
In these cases we need at least $u\in Y^{(1)}$ to make $u_h$ well-defined. 
If we also want to ensure higher regularity of $u_h$, then we have to use the continuous bilinear forms
\begin{align*}
  &\map {g_A(\cdot)[\cdot]}{H^{k+1}(I; V)\times W^{k+1,\infty}(I; \mathcal L(V, V^*))}{H^{k+1}(I; V^*)}, \\
  &\map {g_B(\cdot)[\cdot]}{H^{k+1}(I; H)\times W^{k,\infty}(I; \mathcal L(H))}{H^k(I; H)}, \\
  &\map {g_C(\cdot)[\cdot]}{H^{k+2}(I; H)\times W^{k+1,\infty}(I; \mathcal L(H))}{H^k(I; H)}, \\
  &\map {g_Q(\cdot)[\cdot]}{H^k(I; V)\times W^{k,\infty}(I; \mathcal L(V, H))}{H^k(I; H)},
\end{align*}
resulting in $u_h\in Y^{(k)}$, as long as the operators on the left-hand side of~\eqref{eq:abstract:linearized_general} belong to $X^{(k)}$.
This discussion only yields the existence of $u_h$. 
For the proof that $u_h$ indeed describes the Fréchet-derivative of $S$ we need another ingredient, namely that $S$ is locally Lipschitz continuous.

\begin{theorem}\label{theorem:abstract:lipschitz}
  Let $k\in \N\cup \{0\}$ and $f\in \mathcal F^{(k)}$. 
  If $k\neq 0$ then we also assume $u_0=u_1=0$. 
  Then 
  \begin{enumerate}[(i)]
    \item the map $\map S{D(S) \cap X^{(k)}}{Y^{(k)}}$ is locally Lipschitz continuous in the arguments $B$ and $Q$,
    \item for $k\geq 1$ the map $\map S{D(S) \cap X^{(k)}}{Y^{(k-1)}}$ is locally Lipschitz continuous. 
  \end{enumerate}
\end{theorem}

\begin{proof}\hfill
  \begin{enumerate}[(i)]
    \item The proofs for $Q$ and $B$ are similar, therefore we demonstrate it using $Q$. 
    Let $p=(A, B, C, Q)$, $p^+=(A, B, C, Q^+) \in D(S)\cap X^{(k)}$, $u=S(p)$ and $u^+=S(p^+)$. 
    By subtracting the equations that are solved by $u$ and $u^+$ we conclude that $w\assign u^+-u$ solves
    \begin{equation*}
      (\mathcal Cw')'+\mathcal Bw' + (\mathcal A + \mathcal Q) w = g_Q(u^+)[Q-Q^+]
    \end{equation*}
    in $L^2(I; V^*)$ and possesses homogeneous initial conditions. 
    \Cref{theorem:abstract:regularity} shows that $w$ fulfills the energy estimate 
    \begin{align*}
      \norm{w}_{Y^{(k)}} &\leq \lambda_Q \norm{g_Q(u^+)[Q-Q^+]}_{H^k(I; H)} \leq  \lambda_Q \norm{u^+}_{H^k(I; V)} \norm{Q-Q^+}_{W^{k,\infty}(I; \mathcal L(V, H))} \\
      &\leq  \lambda_Q \lambda_{Q^+} \norm{f} \norm{Q-Q^+}_{W^{k,\infty}(I; \mathcal L(V, H))},
    \end{align*}
    with $\lambda_Q, \lambda_{Q^+}>0$ depending continuously not only on $Q$ and $Q^+$, respectively, but also on the other operators measured in $X^{(k)}$.
    \item If we start the same way with $C$, then we have to be mindful of the initial conditions because they depend on $C$ and $C^+$ ($(\mathcal Cu')(0)=(\mathcal C^+ (u^+)')(0) = 0$). 
    Due to $k\geq 1$ we have $u^+\in H^2(I; H)$, i.e. $(u^+)'$ is continuous (taking values in $H$) and therefore $(\mathcal C(u^+)')(0)=0$ and $(\mathcal Cw')(0) = 0$ also in this case. 
    Energy estimates for $w$ then show 
    \begin{align*}
      \norm{w}_{Y^{(k-1)}} &\leq \lambda_C^{(k-1)} \norm{g_C(u^+)[C-C^+]}_{H^{k-1}(I; H)} \\
      &\leq  \lambda_C^{(k-1)} \norm{u^+}_{H^{k+1}(I; H)} \norm{C-C^+}_{W^{k,\infty}(I; \Lsa(H))} \\
      &\leq  \lambda_C^{(k-1)} \lambda_{C^+}^{(k)} \norm{f} \norm{C-C^+}_{W^{k,\infty}(I; \Lsa(H))},
    \end{align*}
    with constants $\lambda_C^{(k-1)}, \lambda_{C^+}^{(k)} > 0$ depending continuosly on the operators in the $X^{(k-1)}$ and $X^{(k)}$-norm, respectively.
    Estimates for $A$ can be derived in the same fashion. 
    There we have to use $\norm{w}_{Y^{(k-1)}} \leq  \lambda_A^{(k-1)} \lambda_{A^+}^{(k)} \norm{f} \norm{A-A^+}_{W^{k,\infty}(I; \Lsa(V, V^*))}$.
    %
  
    We conclude that $\map S{D(S) \cap X^{(k)}}{Y^{(k-1)}}$ is locally Lipschitz continuous in all arguments with constants that also depend continuously on the other arguments. 
    Therefore the whole map is locally Lipschitz continuous as well.\qedhere
  \end{enumerate}
\end{proof}

Now we can apply this theorem to show differentiability of $S$ in each argument. 

\begin{theorem}\label{theorem:abstract:frechet_partial}
  Let $k\in \N\cup \{0\}$ and $f\in \mathcal F^{(k)}$. 
  If $k\neq 0$ then we also assume $u_0=u_1=0$. 
  Then
  \begin{enumerate}[(i)]
    \item the map $\map S{D(S) \cap X^{(k)}}{Y^{(k)}}$ is Fréchet-differentiable in $B$ and $Q$, and 
    \item for $k\geq 2$ the map $\map S{D(S) \cap X^{(k)}}{Y^{(k-2)}}$ is Fréchet-differentiable in all arguments.
  \end{enumerate}
  For each of these cases and each symbol $x\in \{A, B, C, Q\}$ is $u_h = (\partial_x S)(A, B, C, Q)[h]$ given as the unique solution of the equation 
  \begin{equation*}
      (\mathcal Cu_h')'+\mathcal Bu_h' + (\mathcal A + \mathcal Q) u_h = g_x(u)[h],
  \end{equation*}
  together with homogeneous initial conditions.
\end{theorem}

\begin{proof}\hfill
  \begin{enumerate}[(i)]
    \item Let $p=(A, B, C, Q)\in D(S)\cap X$, $h\in L^\infty(I; \mathcal L(H))$, $u = Sp$, $u^+ = S(p+(0,h, 0,0))$ and $u_h$ as in the assertion. 
    Their difference $w = u^+ - u - u_h$ solves the equation 
    \begin{equation*}
      (\mathcal Cw')'+\mathcal Bw' + (\mathcal A + \mathcal Q) w = g_B(u-u^+)[h]
    \end{equation*}
    in $L^2(I; V^*)$ with vanishing initial conditions. 
    We use energy estimates for $w$ and \Cref{theorem:abstract:lipschitz} to obtain constants $\lambda_B, \lambda_{B+h}$ (that continuously depend on $B$ and $h$) and the estimate
    \begin{align*}
      \norm{w}_Y &\leq \lambda_B \norm{g_B(u-u^+)[h]}_{L^2(I; H)} \leq  \lambda_B \norm{u-u^+}_{H^1(I; H)}  \norm{h}_{L^\infty(I; \mathcal L(H))}\\
      &\leq \lambda_B^2 \lambda_{B+h} \norm{f} \norm{h}^2_{L^\infty(I; \mathcal L(H))} = \mathcal O\left(\norm{h}_{W^{1,\infty}(I; \mathcal L(H))}^2\right),
    \end{align*}
    which shows differentiability of $\map S{D(S)\cap X}Y$ w.r.t. $B$ and can be performed in the same way for $Q$ and the case $\map S{D(S)\cap X^{(k)}}{Y^{(k)}}$.
    \item For $A$ and $C$ we need to use different spaces for the estimation of $w$. 
    For every $h\in W^{k+1}(I; \Lsa(V, V^*))$ with $A+h\in L^\infty(I; \Lsa_{a_0}(V, V^*))$ we calculate
    \begin{align*}
      \norm{w}_{Y^{(k-2)}} &\leq \lambda_A^{(k-2)} \norm{g_A(u-u^+)[h]}_{H^{k-1}(I; V^*)} \\ 
      &\leq \lambda_A^{(k-2)} \lambda_A^{(k-1)} \lambda_{A+h}^{(k)} \norm{f} \norm{h}_{W^{k,\infty}(I; \Lsa(V, V^*))}  \norm{h}_{W^{k-1, \infty}(I; \Lsa(V, V^*))} = \mathcal O(\norm{h}^2).
    \end{align*}
    Note that the last equality holds only if $h$ is measured in the $W^{k+1}(I; \Lsa(V, V^*))$ (or stronger) norm because the constant $\lambda_{A+h}^{(k)}$ depends on the norm of $h$ in this space. 
    Regarding the deriviative in direction $C$: For $h\in W^{k+1}(I; \Lsa(H))$ small enough such that $C+h\in L^\infty(I; \Lsa_{c_0}(H))$ holds we see that $w$ also fulfills homogeneous initial conditions (because $u', (u^+)'$ and $u_h'$ are continuous in $t=0$) and
    \begin{align*}
      \norm{w}_{Y^{(k-2)}} &\leq \lambda_C^{(k-2)} \norm{g_A(u-u^+)[h]}_{H^{k-2}(I; H)} \leq \lambda_C^{(k-2)} \norm{u-u^+}_{H^{k}(I; H)}  \norm{h}_{W^{k-2, \infty}(I; \Lsa(H))}\\
      &\leq \lambda_C^{(k-2)} \lambda_C^{(k-1)} \lambda_{C+h}^{(k)} \norm{f} \norm{h}_{W^{k,\infty}(I; \Lsa(H))} \norm{h}_{W^{k-1, \infty}(I; \Lsa(H))} = \mathcal O(\norm{h}^2).
    \end{align*}
    Again, the last step only holds for $h\in W^{k+1}(I; \Lsa(H))$ because of the constant $\lambda_{C+h}^{(k)}$. \qedhere
  \end{enumerate}
\end{proof}

We would like to remark that although the derivative in direction $A$ or $C$ maps to $Y^{(k-1)}$, we can only show that it is indeed the derivative in the weaker norm of $Y^{(k-2)}$. 
This loss of regularity is due to the application of \Cref{theorem:abstract:lipschitz}. 
This is also the reason why we cannot show the tangential cone condition in these cases. 
On the other hand, the estimate for the linearization error in direction $B$ or $Q$ enables to show the tangential cone condition because there this loss of regularity does not occur (cf.~\cite{gerken_reconstruction_2017}). 

When trying to reconstruct one of the operators $A, B, C$ and $Q$ or a parameter that influences exactly one of these operators, this differentiability result is sufficient. 
If on the other hand the searched for quantity influences multiple operators then we also require the derivative of the whole operator $S$. 
To obtain this we prove that the partial derivatives of $S$ are locally Lipschitz continuous.
This fact is also interesting for the corresponding inverse problems because it allows to conclude ill-posedness of the derivative from ill-posedness of the nonlinear operator (cf.~\cite{hofmann_factors_1994}). 

\begin{lemma}\label{lemma:abstract:frechet_lipschitz}
  Let $k\geq 2$, $f\in \mathcal F^{(k)}$ and $u_0=u_1=0$.
  Each of the operators
  \begin{align*}
    &\partial_A S:\, D(S) \cap X^{(k)} \to \mathcal L(W^{k+1,\infty}(I; \Lsa(V,V^*)), Y^{(k-2)})\\
    &\partial_B S:\, D(S) \cap X^{(k)} \to \mathcal L(W^{k,\infty}(I; \mathcal L(H)), Y^{(k-1)})\\
    &\partial_C S:\, D(S) \cap X^{(k)} \to \mathcal L(W^{k+1,\infty}(I; \Lsa(H)), Y^{(k-2)})\\
    &\partial_Q S:\, D(S) \cap X^{(k)} \to \mathcal L(W^{k,\infty}(I; \mathcal L(V,H)), Y^{(k-1)})
  \end{align*}
  is locally Lipschitz continuous.
\end{lemma}

\begin{proof}
  The proofs only differ in the use of different spaces, and the most difficult ones are $A$ and $C$. 
  Therefore we only demonstrate the proof for $\partial_C$. \\
  For $i=1,2$ let $p_i=(A_i, B_i, C_i, Q_i)\in D(S)\cap X^{(k)}$, $h\in W^{k+1,\infty}(I; \Lsa(H))$, $u_h^{(i)}=\partial_C S(p_i)[h]$ and $u^{(i)}=S(p_i)$.
  The weak formulations of $u_h^{(1,2)}$ differ in their left- and right-hand sides. 
  To connect them, we introduce the function $w_h$ which solves the equation with the left-hand side of $u_h^{(1)}$ and the right-hand side of $u_h^{(2)}$. 
  Thus, in addition to homogeneous initial conditions, $w_h$ solves
  \begin{equation*}
    (\mathcal C_1 w_h')' + \mathcal B_1 w_h' + (\mathcal A_1+\mathcal Q_1) w_h = g_C(u^{(2)})[h]
  \end{equation*}
  in the $L^2(I; V^*)$-sense.
  As noted, $u_h^{(1)}$ and $w_h$ solve the same formulation with a different right-hand side. 
  We apply \Cref{theorem:abstract:regularity} and obtain a constant $\Lambda_1$, depending on $k$ and continuously on the $X^{(k-2)}$-norm of $p_1$, with
  \begin{align*}
    \norm{u_h^{(1)}-w_h}_{Y^{(k-2)}} &\leq \Lambda_1\norm{g_C(u^{(1)}-u^{(2)})[h]}_{H^{k-2}(I; H)}\\
    &\leq \Lambda_1 \norm{g_C} \norms{u^{(1)}-u^{(2)}}_{H^{k}(I; H)} \norm{h}_{W^{k-1,\infty}(I; \Lsa(H))}.
  \end{align*}
  Here $\norm{g_C}$ denotes the norm of $g_C$ in the space
  \begin{equation*}
    \mathcal L\left(H^{k}(I; H),\ \mathcal L\left(W^{k-1,\infty}(I; \mathcal L(H)),\, H^{k-2}(I; H)\right)\right),
  \end{equation*}
  which only depends on $k$. 
  Now we make use of the local Lipschitz continuity of $S:D(S)\cap X^{(k)} \to Y^{(k-1)}$ and obtain another constant $\Lambda_2$, depending on $p_1$ in $X^{(k)}$ and the estimate
  \begin{equation}
    \norm{u_h^{(1)}-w_h}_{Y^{(k-2)}} \leq \Lambda_1 \Lambda_2 \norm{g_C} \norms{p_1-p_2}_{X^{(k)}} \norm{h}_{W^{k-1,\infty}(I; \Lsa(H))}. \label{eq:abstract:frechet_lipschitz:1}
  \end{equation}
  Next we estimate the distance between $u_h^{(2)}$ and $w_h$. 
  Both functions solve a equation with the same right-hand side, but different left-hand sides. 
  Hence, we can apply Lipschitz continuity of the operator $S$ which would arise when the right-hand side would not be $f$, but $g_C(u^{(2)})[h]\in H^{k-2}(I; H)$. 
  Due to linearity of the equation, the norm of the right-hand side has to enter linearly into the Lipschitz constant, therefore through \Cref{theorem:abstract:lipschitz} we get another constant $\Lambda_3$, depending on $k$ and continuously on $p_1$ in $X^{(k-1)}$, such that 
  \begin{align*}
    \norm{u_h^{(2)}-w_h}_{Y^{(k-2)}} &\leq  \Lambda_3 \norm{g_C(u^{(2)})[h]}_{H^{k-1}(I; H)} \norms{p_1-p_2}_{X^{(k-1)}} \\
    &\leq \Lambda_3 \norm{g_C} \norms{p_1-p_2}_{X^{(k-1)}} \norms{u^{(2)}}_{H^{k+1}(I; H)} \norm{h}_{W^{k,\infty}(I; \Lsa(H))}.
  \end{align*}
  This time $\norm{g_C}$ denotes the norm of $g_C$ in
  \begin{equation*}
    \mathcal L\left(H^{k+1}(I; H),\ \mathcal L\left(W^{k,\infty}(I; \Lsa(H)),\, H^{k-1}(I; H)\right)\right).
  \end{equation*}
  Energy estimates for $S:D(S)\cap X^{(k)}\to Y^{(k)}$ from \Cref{theorem:abstract:regularity} provide $\Lambda_4>0$ with
  \begin{equation}
    \norm{u_h^{(2)}-w_h}_{Y^{(k-2)}} \leq \Lambda_3 \Lambda_4 \norm{g_C}  \norms{p_1-p_2}_{X^{(k-1)}} \norm{f} \norm{h}_{W^{k,\infty}(I; \Lsa(H))}, \label{eq:abstract:frechet_lipschitz:2}
  \end{equation}
  where $f$ is measured in the $H^k(I; H)$- or  $H^{k+1}(I; V^*)$ norm. \\
  Finally, we can combine~\eqref{eq:abstract:frechet_lipschitz:1} and~\eqref{eq:abstract:frechet_lipschitz:2} to conclude 
  \begin{align*}
    \norm{u_h^{(1)}-u_h^{(2)}}_{Y^{(k-2)}} &\leq \norm{u_h^{(1)}-w_h}_{Y^{(k-2)}} + \norm{u_h^{(2)}-w_h}_{Y^{(k-2)}} \\
    &\leq \Lambda \norm{f} \norms{p_1-p_2}_{X^{(k-1)}} \norm{h}_{W^{k,\infty}(I; \Lsa(H))}\\
    &\leq \Lambda \norm{f} \norms{p_1-p_2}_{X^{(k-1)}} \norm{h}_{W^{k+1,\infty}(I; \Lsa(H))},
  \end{align*}
  where $\Lambda$ depends continuously on $p_1$ in $X^{(k)}$.
\end{proof}

The differentiability of the whole operator $S$ follows from the differentiability in each arguments and the continuity of the derivatives. 

\begin{corollary}\label{corollary:abstract:frechet_total}
  Let $k\geq 2$, $u_0=u_1=0$ and $f\in \mathcal F^{(k)}$.  
  The operator $S: D(S) \cap X^{(k)} \to Y^{(k-2)}$ is Fréchet differentiable in every $p=(A, B, C, Q)\in D(S)\cap X^{(k)}$. 
  Its derivative $\partial S(p)[h]$ is given for every $h=(\bar A, \bar B, \bar C, \bar Q)\in X^{(k)}$ as the solution $u_h$ of
  \begin{align*}
      (\mathcal Cu_h')'+\mathcal Bu_h' + (\mathcal A + \mathcal Q) u_h &= g_A(u)[\bar A] + g_B(u)[\bar B] + g_C(u)[\bar C] + g_Q(u)[\bar Q] \\ &= -(\bar A(\cdot)+\bar Q(\cdot))[u(\cdot)] - \bar B(\cdot)[u'(\cdot)] - (\bar C(\cdot)[u'(\cdot)])',
  \end{align*}
  that has vanishing initial conditions $u_h(0)= (\mathcal C u_h')(0) = 0$. 
  As always, $u=S(p)$ denotes the solution of the direct problem. 
  Furthermore, the map $\partial S: D(S) \cap X^{(k)} \to \mathcal L(X^{(k)}, Y^{(k-2)})$ is locally Lipschitz continuous. 
\end{corollary}

\subsection{Adjoint of the Fréchet-Derivative}%
\label{section:abstract:adjoint}

For the numerical inversion of linearized problems that arise from $S$ we need not only its Fréchet derivative, but also its adjoint. 
At this point we only know that this adjoint exists, but have no means of calculating it efficiently. 
From an application viewpoint, $Y$ (or even $Y^{(k)}$) is not a suitable space for the (measured) data that is presumed to be noisy because this would imply that the noise is differentiable in time. 
An approach with $L^2$-spaces seems more sensible here, and also makes the analysis easier because $L^2(I; H)$ is a Hilbert space. 
Therefore we seek to calculate the adjoint of $\partial S(p)\in \mathcal L(X^{(k)}, L^2(I; H))$, which can be identified with an operator $\partial S(p)^* \in \mathcal L(L^2(I; H), (X^{(k)})^*)$. 
But even for this choice in spaces, the application of $\partial S(p)^*[v]\in {(X^{(k)})}^*$ to $h\in X^{(k)}$ must still be calculated by $\scp v{\partial S(p)h}_{L^2(I; H)}$ and therefore requires the solution of a different PDE for every $h$. 
Hence, we will try to shift as many operations from $h$ to $v$ as possible.

Unsurprisingly, $\partial S(p)^*[v]$ will involve the solution of an evolution equation, namely of the adjoint equation to~\eqref{eq:abstract:problem}, i.e. 
\begin{equation}\label{eq:abstract:adjoint_equation}
  (\mathcal Cw^\prime)^\prime - \mathcal B^* w^\prime + (\mathcal A + \mathcal Q^* - (\mathcal B^*)^\prime) w = v \ \text{in $L^2(I; V^*)$.}
\end{equation}
Here $\mathcal B^*$ and $\mathcal Q^*$ denote the realization of $t\mapsto B^*(t)$ and $Q^*(t)$, respectively. 
Due to the pointwise definition these are identical to the adjoints of the realizations $\mathcal B\in \mathcal L(L^2(I; H))$ and $\mathcal Q \in \mathcal L(L^2(I; H), L^2(I; V^*))$. 
Equation~\eqref{eq:abstract:adjoint_equation} has to be equipped with homogeneous \emph{end conditions} $w(T) = (\mathcal Cw^\prime)(T) = 0$, and in this form is the adjoint of the operator $f\mapsto u$ with $f$ and $u$ as in~\eqref{eq:abstract:problem} with respect to $L^2(I; H)$.
If $B=0$ and $Q$ is pointwise self-adjoint, then this is the original equation which has to be solved backwards in time. 
In any case, conditions for the unique solvability and regularity are also given by \Cref{lemma:abstract:wellposed} and \Cref{theorem:abstract:regularity} after reversing time using the transformation $t\mapsto T-t$.   

\begin{theorem}\label{theorem:abstract:adjoints}
Let $p=(A, B, C, Q)\in D(S)\cap X^{(k)}$ with $Q^*\in L^\infty(I; \mathcal L(V, H))$, $k\in \N\cup \{0\}$, $f\in \mathcal F^{(k)}$ and $u=S(p)$. 
In the case that $k\geq 1$ we also assume $u_0=u_1=0$.
For $v\in L^2(I; H)$ let $w_v$ denote the solution of~\eqref{eq:abstract:adjoint_equation} with homogeneous end conditions.
\begin{enumerate}[(i)]
  \item Set $k_1=\max \{k,1\}$. 
  The adjoints of 
  \begin{equation*}
    \partial_Q S(p) \in \mathcal L(W^{k_1,\infty}(I; \mathcal L(V, H)), L^2(I; H)) \text{ and } \partial_B S(p) \in \mathcal L(W^{k_1,\infty}(I; \mathcal L(H)), L^2(I; H))
  \end{equation*}
  can be characterized for $\bar Q \in W^{k_1,\infty}(I; \mathcal L(V, H))$, $\bar B \in W^{k_1,\infty}(I; \mathcal L(H))$ via
  \begin{align*}
    \dup{(\partial_Q S(p))^*[v]}{\bar Q}_{W^{k_1,\infty}(I; \mathcal L(V, H))^* \times W^{k_1,\infty}(I; \mathcal L(V, H))} &= -\int_0^T \scp{\bar Q(t)u(t)}{w_v(t)} \dt  \\
    \text{and } \dup{(\partial_B S(p))^*[v]}{\bar B}_{W^{k_1,\infty}(I; \mathcal L(H))^* \times W^{k_1,\infty}(I; \mathcal L(H))} &= -\int_0^T \scp{\bar B(t)u'(t)}{w_v(t)} \dt.
  \end{align*}
  \item If $k\geq 2$, then the evaluation of the adjoints of
  \begin{equation*}
    \partial_A S(p) \in \mathcal L(W^{k+1,\infty}(I; \Lsa(V, V^*)), L^2(I; H)) \text{,\ } \partial_C S(p) \in \mathcal L(W^{k+1,\infty}(I; \Lsa(H)), L^2(I; H))
  \end{equation*}
  can be expressed for every $\bar C\in W^{k+1,\infty}(I; \Lsa(H))$, $\bar A\in W^{k+1,\infty}(I; \Lsa(V, V^*))$ through
  \begin{align*}
    \dup{(\partial_A S(p))^*[v]}{\bar A}_{W^{k+1,\infty}(I; \Lsa(V, V^*))^* \times W^{k+1,\infty}(I; \Lsa(V, V^*))} &= -\int_0^T \dup{\bar A(t)u(t)}{w_v(t)} \dt \\
    \text{and } \dup{(\partial_C S(p))^*[v]}{\bar C}_{W^{k+1,\infty}(I; \Lsa(H))^* \times W^{k+1,\infty}(I; \Lsa(H))} &= \int_0^T \scp{\bar C(t) u'(t)}{w_v'(t)} \dt. 
  \end{align*}
\end{enumerate}
\end{theorem}

\begin{proof}
  The assumption $t\mapsto Q(t)^* \in L^\infty(I; \mathcal L(V, H))$ has to be made to guarantee existence of $w=w_v\in Y$. 
  Then \Cref{lemma:abstract:wellposed} states that $w_v$ is well-defined for all $v\in L^2(I; H)$ and depends continuously on $v$. 
  We continue by verifying that $w_v$ indeed has something to do with the adjoint of $\partial S$. 
  From \Cref{theorem:abstract:frechet_partial} we know that for each symbol $x\in \{A, B, C, Q\}$ holds $\partial_x S(p)[h] = u_h \in Y$, where $u_h$ solves 
  \begin{equation*}
    (\mathcal Cu_h^\prime)^\prime + \mathcal Bu_h^\prime + (\mathcal A+\mathcal Q)u_h = g_x(u)[h]
  \end{equation*}
  with $u = S(p)$. 
  We test the equation that is solved by $w_v$ at time $t\in I$ with $u_h(t)$ and integrate over $I=(0,T)$ to attain
  \begin{equation}\label{eq:abstract:adjoints:1}
    \begin{aligned}
      \scp{v}{u_h}_{L^2(I; H)} = \int_0^T &\dup{(\mathcal Cw_v^\prime)^\prime(t)}{u_h(t)} + \dup{A(t)w_v(t)}{u_h(t)} \\
      +  &\scp{Q^*(t)w_v(t)}{u_h(t)} - \scp{B^*(t)w_v^\prime(t)}{u_h(t)} - \scp{(B^*)^\prime(t)w_v(t)}{u_h(t)} \dt.
    \end{aligned}
  \end{equation}
  On the first expression in the integral on the right-hand side of~\eqref{eq:abstract:adjoints:1} we apply the integration by parts formula in the Gelfand triple $H\subset V^* \subset H^*$ (cf.~\cite{zeidler_nonlinear_1985}) and conclude
  \begin{align}
    \int_0^T \dup{(\mathcal Cw_v^\prime)^\prime(t)}{u_h(t)} \!\dt &= \scp{(\mathcal C w_v^\prime)(T)}{u_h(T)} - \scp{(\mathcal C w_v^\prime)(0)}{u_h(0)} - \int_0^T \scp{(\mathcal Cw_v^\prime)(t)}{u_h^\prime(t)} \!\dt \notag \\
    &= - \int_0^T \scp{(\mathcal C u_h^\prime)(t)}{w_v^\prime(t)} \dt \notag \\
    &= \scp{(\mathcal C u_h^\prime)(0)}{w_v(0)}-\scp{(\mathcal C u_h^\prime)(T)}{w_v(T)} + \int_0^T \scp{(\mathcal C u_h^\prime)^\prime(t)}{w_v(t)} \!\dt \notag \\
    &= \int_0^T \scp{(\mathcal C u_h^\prime)^\prime(t)}{w_v(t)} \dt. \label{eq:abstract:adjoints:2}
  \end{align}
  For expressions in~\eqref{eq:abstract:adjoints:1} that involve $B$ we can apply the product rule (Lemma 2.2 in~\cite{gerken_dynamic_2018}) to see
  \begin{align}
    \int_0^T &\scp{B^*(t)w_v^\prime(t)}{u_h(t)} + \scp{(B^*)^\prime(t)w_v(t)}{u_h(t)} \dt = \int_0^T \scp{(\mathcal B^* w_v)^\prime(t)}{u_h(t)} \dt \notag \\
    &= \scp{B^*(T) w_v(T)}{u_h(T)} - \scp{B^*(0) w_v(0)}{u_h(0)} - \int_0^T \scp{B^*(t) w_v(t)}{u_h^\prime(t)} \dt \notag \\
    &= - \int_0^T \scp{B(t)u_h^\prime(t)}{w_v(t)} \dt. \label{eq:abstract:adjoints:3}
  \end{align}
  Dealing with $Q$ and $A$ is simple because we only need to insert their adjoints, bearing in mind that $A(t)$ is self-adjoint. 
  Now we can use~\eqref{eq:abstract:adjoints:2} and~\eqref{eq:abstract:adjoints:3} in~\eqref{eq:abstract:adjoints:1} and obtain 
  \begin{align*}
    \scp{v}{u_h}_{L^2(I; H)} = \int_0^T &\dup{(\mathcal Cu_h^\prime)^\prime(t)}{w_v(t)} + \scp{B(t)u_h^\prime(t)}{w_v(t)} \\
    &+ \dup{A(t)u_h(t)}{w_v(t)} + \scp{Q(t)u_h(t)}{w_v(t)} \dt,
  \end{align*}
  which contains the left-hand side of the equation that is solved by $u_h(t)$, tested with $w_v(t)$. 
  We replace it by the corresponding right-hand side and arrive at
  \begin{equation*}
    \dup{\partial_x S(p)^*[v]}h = \scp{v}{\partial_x S(p)[h]}_{L^2(I; H)} = \scp{v}{u_h}_{L^2([0,T], H)} = \int_0^T \dup{g_x(u)[h](t)}{w_v(t)} \dt.
  \end{equation*}
  The assertion follows by stating the correct spaces for $h$ and the definition of $g_x$. 
  In the case of $\partial_C S(p)^*$ we can use the integration by parts formula once more to get rid of the time derivative on $h$. 
\end{proof}


With this result the application of $\partial S(p)^*[v]$ on $h\in X^{(k)}$ can be implemented efficiently because the effort of computing $w_v$ does not depend on $h$, and the operations that do depend on $h$ (multiplication, integration over $I$) are cheap. 
Unfortunately, we are not able to represent the adjoint completely using the $L^2(I; H)$ dot product because we do not know how the application e.g.\ of $\bar C(t)$ on $u'(t)$ looks like. 
This will be the case in sections~\ref{section:elasticity} and~\ref{section:maxwell}, where we apply this theory to actual PDEs and therefore have more information about the structure of the operators. 

As a direct consequence of the above theorem we can also describe the adjoint of $\partial S(p)$.

\begin{corollary}\label{corollary:abstract:adjoint_total}
  Let the assumptions of \Cref{theorem:abstract:adjoints} be fulfilled with $k\geq 2$. 
  The adjoint $(\partial S(p))^*\in \mathcal L(L^2(I; H), (X^{(k)})^*)$ of $\partial S(p) \in \mathcal L(X^{(k)}, L^2(I; H))$ at $v\in L^2(I; H)$, $h = (\bar A, \bar B, \bar C, \bar Q) \in X^{(k)}$ is given by 
  \begin{align*}
    \dup{(\partial S(p))^*[v]}{h}_{(X^{(k)})^* \times X^{(k)}} = \int_0^T &\scp{\bar C(t) u'(t)}{w_v'(t)} - \scp{\bar B(t)u'(t)}{w_v(t)} \\
    &- \dup{\left(\bar A(t)+\bar Q(t)\right)u(t)}{w_v(t)}  \dt.
  \end{align*}
\end{corollary}

\subsection{Ill-posedness}

In particular for the numerical treatment of inverse problems it is important to know whether the task under consideration is ill-posed or not because this fact has a large impact on the applicable algorithms. 
Therefore we will discuss the ill-posedness of $S$ and also its linearization. 

We do not prove the (local) ill-posedness of $S$ directly, but formulate an intermediate result first that can also be used to show ill-posedness in a setting where not the operators themselves, but another parameter that influences them is sought. 
In this case it is important that the perturbations which have been used to show ill-posedness of $S$ lie in the image of the operator that maps searched-for parameters to the operators $A, B, C$ and $Q$.
For both situations we need to be aware in which circumstances the image of a sequence of parameters under $S$ converges.

\begin{theorem}\label{theorem:abstract:illposed_convergence}
Let $k\in \N\cup \{0\}$ and $f\in \mathcal F^{(k)}$.
If $k\geq 1$ then we also require $u_0=u_1=0$. 
Further let $p=(A, B, C, Q)\in D(S) \cap X^{(k)}$, $u= S(p)$ and $k_1=\max \{k,1\}$.
\begin{enumerate}[(i)]
    \item If $(R_j)_{j\in\N} \subset W^{k_1,\infty}(I; \mathcal L(V, H))$ satisfies $\norm{R_j}\leq \Gamma$ and $\mathcal R_j v \to 0$ in $H^k(I; H)$ for all $v\in Y^{(k)}$, then $S(A, B, C, Q+R_j)\to u$ in $Y^{(k)}$ when $j\to\infty$.
    \item If $(R_j)_{j\in\N} \subset W^{k_1,\infty}(I; \mathcal L(H))$ satisfies $\norm{R_j}\leq \Gamma$ and $\mathcal R_j v' \to 0$ in $H^k(I; H)$ for all $v\in Y^{(k)}$, then $S(A, B+R_j, C, Q)\to u$ in $Y^{(k)}$ when $j\to\infty$.
    \item Let $k>0$ and $(R_j)_{j\in\N} \subset W^{k+1,\infty}(I; \mathcal L(V, V^*))$ with $\norm{R_j}\leq \Gamma$, with $\Gamma$ small enough to guarantee $(A+R_j, B, C, Q)\in D(S)$ for all $j$, and $\mathcal R_j v \to 0$ in $H^{k}(I; V^*)$ for all $v\in Y^{(k)}$.
    Then $S(A+R_j, B, C, Q)\to u$ in $Y^{(k-1)}$ when $j\to\infty$.
    \item Let $k>0$ and $(R_j)_{j\in\N} \subset W^{k+1,\infty}(I; \mathcal L(H))$ with $\norm{R_j}\leq \Gamma$, with $\Gamma$ small genug enough to guarantee $(A, B, C+R_j, Q)\in D(S)$ for all $j$, and $(\mathcal R_j v')' \to 0$ in $H^{k-1}(I; H)$ for all $v\in Y^{(k)}$.
    Then $S(A, B, C+R_j, Q)\to u$ in $Y^{(k-1)}$ when $j\to\infty$.
\end{enumerate}
In each case the convergence is uniform in $(A, B, C, Q)$ on every bounded subset of $D(S)\cap X^{(k)}$.
\end{theorem}

\begin{proof}\hfill
  \begin{enumerate}[(i)]
    \item We start with $Q$. 
    Let $u_j = S(A, B, C, Q+R_j)$. 
    The fields $u$ and $u_j$ solve 
    \begin{align*}
      (\mathcal Cu')' + \mathcal Bu' + (\mathcal Q+ \mathcal R_j)u + \mathcal Au &= f + \mathcal R_j u \\
      (\mathcal Cu_j')' + \mathcal Bu_j' + (\mathcal Q + \mathcal R_j)u_j + \mathcal Au_j &= f
    \end{align*}
    with the same initial conditions. 
    Hence, $w_j=u-u_j$ is a solution to
    \begin{equation*}
      (\mathcal Cw_j')' + \mathcal Bw_j' + (\mathcal Q+\mathcal R_j)w_j + \mathcal Aw_j = \mathcal R_j u 
    \end{equation*}
    with homogeneous initial conditions and satisfies
    \begin{equation}\label{eq:abstract:illposed_convergence:1}
      \norm{w_j}_{Y^{(k)}}^2 \leq \lambda_j^2 \norm{\mathcal R_j u}_{H^k(I; H)}^2
    \end{equation}
    with $\lambda_j > 0$ that stays bounded when $j\to\infty$ (the constants in the energy estimates are continuous and the $R_j$ are bounded). 
    From the properties of $R_j$ we deduce $\norm{w_j}\to 0$ when $j\to\infty$. 
    This convergences is uniform in $A, B, C$ and $Q$ because both $\lambda_j$ and $u$ depend continuously on them.
    \item The proof for $B$ can be done in the same fashion, instead of~\eqref{eq:abstract:illposed_convergence:1} we obtain 
    \begin{equation*}
      \norm{w_j}_{Y^{(k)}}^2 \leq \lambda_j^2 \norm{\mathcal R_j u'}_{H^k(I; H)}^2 \to 0.
    \end{equation*}
    \item For the other two operators we lose one order of regularity because the right-hand side of the equation that is solved by $w_j=u-u_j$ is less regular. 
    When we perturb $A$ the $w_j$ satisfy
    \begin{equation*}
      \norm{w_j}_{Y^{(k-1)}}^2 \leq \lambda_j^2 \norm{\mathcal R_j u}_{H^{k}(I; V^*)}^2,
    \end{equation*}
    which vanishes in the limit $j\to\infty$.
    \item In the case of $C$ the estimate reads $\norm{w_j}_{Y^{(k-1)}}^2 \leq \lambda_j^2 \norm{(\mathcal R_j u')'}_{H^{k-1}(I; H)}^2 \to 0$. \qedhere
  \end{enumerate}
\end{proof}

The uniform convergence is important when a searched for quantity influences not only one, but multiple operators.
Now we show that such sequences $R_j$ always exist (even in this general framework), and conclude that the reconstruction of the operators is indeed an ill-posed problem. 

\begin{lemma}\label{lemma:abstract:illposed_helper}
  There exist constants $\Gamma > \gamma > 0$ and sequences of operators
  \begin{enumerate}[(i)]
    \item $(X_k)_{k\in\N} \subset \Lsa(H)$ such that $X_k v \to 0$ in $H$ for all fixed $v\in H$ and $\Gamma \geq \norm{X_k} \geq \gamma $ in $\Lsa(H)$ and $\Lsa(V, V^*)$,
    \item $(Y_k)_{k\in\N} \subset \mathcal L(V)$ with $Y_k v \to 0$ in $V$ for all fixed $v\in V$ and $\Gamma \geq \norm{Y_k}\geq \gamma$ in $\mathcal L(V)$ and $\mathcal L(V, H)$.
  \end{enumerate}
\end{lemma}

\begin{proof}
  From the pointwise convergence (and therefore boundedness) of the operators we can already deduce the existence of the upper bound $\Gamma$ using the uniform boundedness principle. 
  \begin{enumerate}[(i)]
    \item Let $(\phi_j)_{j\in\N}\subset V$ denote an orthonormal basis of $H$ (possible because $V$ is dense in $H$). 
    We use it to define $X_k$ for $v\in H$ as
    \begin{equation*}
      X_k v = \scp{v}{\phi_k}_H \phi_k.
    \end{equation*}
    Apparently $\norm{X_k v}_H = |\scp{v}{\phi_k}_H| \to 0$ for $v\in H$ and $\norm{X_k}_{\mathcal L(H)}\leq 1$. 
    By evaluating $X_k$ at $v=\phi_k$ we can also see $\norm{X_k}\geq 1$.
    For $u, v\in V$ the identity $\dup{X_k v}{u}_{V^*\times V} = \scp v{\phi_k}_H \dup{\phi_k}{u}_{V^*\times V} = \scp v{\phi_k}_H \scp{u}{\phi_k}_H$ holds, which implies $\norm{X_k}_{\mathcal L(V, V^*)} \leq 1$ and due to $\phi_k\in V$ we may set $u=v=\phi_k$ to infer $\norm{X_k}_{\mathcal L(V, V^*)} \geq 1$. 
    \item For $\mathcal L(V, V)$ we could use the same $X_k$ if we replace $\phi_k$ by an orthonormal basis $(\psi_k)_{k\in\N}$ of $V$, but this sequence would not be suitable for $\mathcal L(V, H)$. 
    This is due to compactness of $V$ in $H$ since ONBs of $V$ convergence strongly to zero in $H$ and $V^*$. 
    Hence, we modify the definition slightly to arrive at
    \begin{equation*}
      Y_k v = \scp{v}{\psi_k}_V \psi_1,
    \end{equation*}
    which works in this case because $Y_k v \to 0$ for $v\in V$, but at the expense that $Y_k$ is not self-adjoint. \qedhere
  \end{enumerate}
\end{proof}

Finally, we can show local ill-posedness of $S$, even with data in $Y^{(k)}$. 
Of course this implies the ill-posedness in the case of data belonging to $L^2(S; H)$ because of the weaker norm.

\begin{theorem}\label{theorem:abstract:illposed_locally}
Let $k\in \N\cup \{0\}$ and $f\in \mathcal F^{(k)}$.
If $k\geq 1$ then we assume $u_0=u_1=0$. 
Let $(A, B, C, Q)\in D(S) \cap X^{(k)}$.
  \begin{enumerate}[(i)]
    \item The tasks of finding $B$ or $Q$ such that $S(A, B, C, Q)=y\in Y^{(k)}$ holds are locally ill-posed in every $B$ and $Q$.
    \item For all $k\in\N$ the tasks of finding $A$ or $C$ such that $S(A, B, C, Q)=y\in Y^{(k-1)}$ holds are locally ill-posed in every $A$ and $C$.
\end{enumerate}
\end{theorem}

\begin{proof}
  We prove the claim by explicitly constructing sequences of operators that do not converge, but stay arbitrary close to $p=(A, B, C, Q)$ such that their image under $S$ converges to $S(p)$. 
  Let $r>0$ be fixed.  
  \begin{enumerate}[(i)]
    \item We start with $Q$ and set $Q_j(t)=Q(t)+\tilde r Y_j$ with $\tilde r \assign r / \Gamma$ and $\Gamma>0$, $Y_j\in \mathcal L(V, H)$ as in \Cref{lemma:abstract:illposed_helper}.
    This way $Q_j\in B(Q, r)$ and $Q_j\not\to Q$ in $W^{k,\infty}(I; \mathcal L(V, H))$.
    We show that $R_j(t)\assign \tilde r Y_j$ satisfies the requirements of \Cref{theorem:abstract:illposed_convergence} (i).
    For $v\in Y^{(k)}$ we have
    \begin{equation}\label{eq:abstract:illposed_locally:1}
      \norm{\tilde r\mathcal Y_j v}_{H^k(I; H)}^2 = \tilde r^2 \sum_{i=0}^k \int_0^T \norms{Y_j v^{(i)}(t)}_{H}^2 \dt.
    \end{equation}
    Every one of the finitely many integrands converges pointwise to zero and is bounded by $\Gamma^2 \norm{v^{(i)}(t)}_V^2 \in L^1(I)$, hence the whole sum vanishes in the limit $j\to\infty$. \\
    For $B$ we have to use $(X_j)_{j\in\N}\subset \mathcal L(H)$ from \Cref{lemma:abstract:illposed_helper} as the perturbation. 
    Instead of~\eqref{eq:abstract:illposed_locally:1} we obtain 
    \begin{equation*}
      \norm{\tilde r\mathcal X_j v'}_{H^k(I; H)}^2 = \tilde r^2 \sum_{i=0}^k \int_0^T \norms{X_j v^{(i+1)}(t)}_{H}^2 \dt,
    \end{equation*}
    which converges to zero for similar reasons. 
    The convergence of $S(p_j)$ to $S(p)$ then follows from \Cref{theorem:abstract:illposed_convergence} (ii).
    \item We set $A_j(t) = A(t) + \tilde r X_j$ and $C_j(t) = C(t) + \tilde r X_j$, still with $\tilde r=r / \Gamma $. 
    Since $D(S)$ is open the resulting $p_j$ belong to $D(S)$ as long as $r$ is sufficiently small. 
    For every $v\in Y^{(k-1)}$ we have
    \begin{equation*}
       \norm{\tilde r\mathcal X_j v}_{H^{k}(I; V^*)}^2 = \tilde r^2 \sum_{i=0}^{k} \int_0^T \norms{X_j v^{(i)}(t)}_{V^*}^2 \dt,
    \end{equation*}
    which converges to zero in the limit.
    For $C$ the reasoning is similar, we obtain
    \begin{equation*}
      \norm{(\tilde r\mathcal X_j v')'}_{H^{k}(I; H)}^2 = \tilde r^2 \sum_{i=0}^{k} \int_0^T \norms{X_j v^{(i+2)}(t)}_{H}^2 \dt
    \end{equation*}
    and use that $v\in H^{k+2}(I; H)$. 
    In both cases we can apply \Cref{theorem:abstract:illposed_convergence}.\qedhere
  \end{enumerate}
\end{proof}

For convenience we used sequences of perturbations that are time independent, which confirms that the corresponding \enquote{static} problems are ill-posed as well. 
In the case of time-dependent functions we would have to ensure that they are smooth enough to belong to $X^{(k)}$, which requires some work. 
We will showcase this when we apply the abstract results to the elastic wave equation. 

Ill-posedness of the linearized problem can be concluded from the local ill-posedness of $S$ because we showed its (local) Lipschitz continuity in \Cref{lemma:abstract:frechet_lipschitz}, but we can also show it directly using compact embeddings, which are established in the following lemma.


\begin{lemma}
  Given $1\leq p\leq \infty$, $1\leq q\leq p$ with $q < \infty$ the embeddings $W^{k,p}(I; V)\cap W^{k+1,p}(I; H)\hookrightarrow W^{k,q}(I; H)$ and $W^{k,\infty}(I; V)\cap W^{k+1,\infty}(I; H)\hookrightarrow C^k(I; H)$ are compact. 
\end{lemma}

\begin{proof}
  Follows by induction from the Aubin-Lions Lemma, see~\cite{aubin_theoreme_1963}.
\end{proof}

We apply this to the derivatives of $S$.

\begin{lemma}\label{lemma:abstract:frechet_compact}
  Let $k\in \N\cup \{0\}$, $f\in \mathcal F^{(k)}$ and $u_0=u_1=0$ if $k\geq 1$. 
  Further, let $p \in D(S) \cap X^{(k)}$ and $k_1 \assign \max \{k,1\}$.
  \begin{enumerate}[(i)]
    \item For $\map S{D(S) \cap X^{(k)}}{Z}$ with $Z=W^{j,p}(I; H)$ or $Z=C^j(I; H)$ with $0\leq j \leq k$ and $1\leq p< \infty$ the derivatives $\partial_Q S(p) \in \mathcal L(W^{k_1,\infty}(I; \mathcal L(V, H)), Z)$ and $\partial_B S(p)\in\mathcal L(W^{k_1,\infty}(I; \mathcal L(H)), Z)$ are compact operators.
    \item If $k\geq 2$ and $\map S{X^{(k)}}{Z}$ with $Z=W^{j,p}(I; H)$ or $Z=C^j(I; H)$ with $0\leq j \leq k-1$ and $1\leq p< \infty$ the operators $\partial_A S(p)\in \mathcal L(W^{k+1,\infty}(I; \Lsa(V, V^*)), Z)$ and
    $\partial_C S(p)\in \mathcal L(W^{k+1,\infty}(I; \Lsa(H)), Z)$ are compact.
  \end{enumerate}
\end{lemma}

\begin{proof}\hfill
  \begin{enumerate}[(i)]
    \item Follows from the compactness of $Y^{(k)}\hookrightarrow Z$.
    \item Note that $Y^{(k-2)}$ is continuously embedded in $Z$, i.e. $\map S{X^{(k)}}{Z}$ is Fréchet-dif\-fer\-en\-tia\-ble w.r.t.\ $A$ and $C$. 
    Additionally $\partial_A S(p)$ and $\partial_C S(p)$ map into $Y^{(k-1)}$, which has an compact embedding into $Z$.\qedhere
  \end{enumerate}
\end{proof}

From the compactness of the derivatives we know that the linearized problems arising from $S$ would be locally ill-posed at every point, but they might still be well-posed by restricting the problem to $N(\partial_x S(p))^\perp$. 
We show that this is not the case.

\begin{lemma}\label{lemma:abstract:frechet_range}
  Assume everything as in \Cref{lemma:abstract:frechet_compact} and additionally that $f\neq 0$. 
  In this setting the range of the following operators is infinite-dimensional for every $p\in D(S) \cap X^{(k)}$:
  \begin{enumerate}[(i)]
    \item $\partial_Q S(p) \in \mathcal L(W^{k_1,\infty}(I; \mathcal L(V, H)), Y^{(k)})$,
    \item $\partial_A S(p) \in \mathcal L(W^{k+1,\infty}(I; \Lsa(V, V^*)), Y^{(k-1)})$ if $k\geq 2$,
    \item $\partial_B S(p) \in \mathcal L(W^{k,\infty}(I; \mathcal L(H)), Y^{(k)})$ if $k\geq 1$ and
    \item $\partial_C S(p) \in \mathcal L(W^{k+1,\infty}(I; \Lsa(H)), Y^{(k-1)})$ if $k\geq 2$.
  \end{enumerate}
\end{lemma}

\begin{proof}\hfill
\begin{description}
  \item[\normalfont(i)\, \& \,(ii):] Assume that one of the operators had a finite dimensional range, i.e.\ that $u_h=\partial_x S(p)[h]$ (with $x=A$ or $x=Q$) can be represented as a finite sum independent of $h\in W^{k,\infty}(I; \mathcal L(V, H))$ and $h\in W^{k+1,\infty}(I; \Lsa(V, V^*))$, respectively. 
  Due to the linearity of the equation solved by $u_h$ its left- and therefore also its right-hand side $-h[u]$ could be written as a finite sum as well. \\
  Since $f\neq 0$ we also have $u=S(p)\neq 0$ and even for $k=0$ know $u\in C(I; H)$. 
  Therefore there exists $t_0\in (0,T)$ and $\epsilon > 0$ such that $t_0+(-\epsilon, \epsilon)\subset (0,T)$ and $u(t_0+s)\neq 0$ for all $s\in (-\epsilon,\epsilon)$.
  Given any sequence of pointwise disjoint balls $B(t_i, \epsilon_i)\subset t_0+(-\epsilon, \epsilon)$ and functions $(\alpha_i)_{i\in\N}\subset C^\infty(\R)$ with $\emptyset \neq \spt \alpha_i \subset B(t_i, \epsilon_i)$ we define $h_i(t) = \alpha_i(t) \mathrm{Id}_H$.
  This way we get $h\in C^{\infty}(I; \mathcal L(H)) \subset C^{\infty}(I; \Lsa(V, V^*))$. 
  The supports of $-h_i[u]$ are non-empty and pairwise disjoint. Hence, the set $\{-h_i[u]\}_{i\in\N}$ is infinite and linear independent, which contradicts the assumption.
  \item[\normalfont(iii)\, \& \,(iv):] Here $h$ is applied to $u'$. 
  For $u'\in C(I; H)$ we have to require regularity with $k\geq 1$. 
  Due to $u_0=0$ and $u\neq 0$ we conclude $u'\neq 0$ and can proceed as in the first part of the proof and obtain $\dim(\mathcal R(\partial_B S(p)))=\infty$.
  When looking at $\partial_C S$ we additionally have to choose $\alpha_i$ in such a way that $(\alpha_i u')' = \alpha_i' u' + \alpha_i u'' \neq 0$, but this is no problem when $k\geq 2$. \qedhere
\end{description}
\end{proof}

\section{Application to Linear Elasticity}%
\label{section:elasticity}

As a first example we consider the propagation of elastic waves through a bounded domain $\Omega \subset \R^3$ in the finite time interval $I=[0,T]$ with $T>0$.
Our model for the displacement field $u : I\times \Omega \to \R^3$ is given through the equation
\begin{equation}\label{eq:elasticity:pde}
  (\rho u')' = \divv \sigma(u) + f\quad \text{in $I\times \Omega$}.
\end{equation}
The right-hand side consists of the restoring force, which is equal to the row-wise divergence of the stress tensor $\sigma(u):I\times \Omega \to \R^{3\times 3}$ due to Hooke's law. 
We also allow for a volumetric force $f : I\times \Omega \to \R$.
The function $\rho$ denotes the mass density inside $\Omega$.

We assume that $\Omega$ consists of a linear isotropic material such that the stress tensor has the form $\sigma(u)=\sigma_{\lambda, \mu}(u)=2\mu \epsilon(u)+ \lambda \divv(u) I_3$. 
The functions $\lambda$ and $\mu$ denote the Lamé coefficients of the material and $I_3$ is the $3\times 3$ unit matrix. 
The symmetric strain tensor $\epsilon(u)=(Du + Du^\top ) / 2$ depends on the Jacobian $Du$ of $u$.

For simplicity we make the assumption that the material is at rest at $t=0$, i.e. $u(0)=u'(0)=0$, and that the body is fixed throughout the whole time. 
This is modeled by the homogeneous Dirichlet boundary condition $u=0$ on $S\times \partial \Omega$. 
This setup implies that the excitation of waves inside $\Omega$ happens only due to the volumetric force $f$.

The inverse problem we would like to consider is the identification of the density $\rho$ and the Lamé coefficients $\lambda$, $\mu$ from measurements of the displacement field $u$. 
This setting is relevant e.g.\ for non-destructive testing, where a deviation in these values might indicate a defect in the material.

We start by stating the elastic equation in the required abstract setting. 
As stated above, we consider the initial boundary value problem 
\begin{subequations}\label{eq:elasticity:problem_pde}
\begin{align}
  (\rho u')' - \divv \sigma(u) &= f\quad \text{in $S\times \Omega$} \\
  u(0) = u'(0) &= 0 \ \text{ in $\Omega$} \\
  u &= 0 \ \text{ in $S\times\partial\Omega$}.
\end{align}
\end{subequations}
\newcommand{\spaceH}{\ensuremath{L^2(\Omega, \R^3)}}
\newcommand{\spaceV}{\ensuremath{H^1_0(\Omega, \R^3)}}
\newcommand{\spaceVdual}{\ensuremath{H^{-1}(\Omega, \R^3)}}
Due to the boundary conditions the appropriate function spaces for the weak formulation are given through $H=\spaceH$, $V=\spaceV$ and therefore we have $V^*=\spaceVdual$. 
A formal integration by parts shows that the weak formulation of the PDE then reads as 
\begin{equation}
  \int_\Omega (\rho(t)u'(t))' v \dx + \int_\Omega \sigma_{\lambda(t), \mu(t)}(u(t)):\epsilon(v) \dx = \int_\Omega f(t)v\dx,
\end{equation}
which should hold for all $v\in \spaceV$ for almost all $t\in I$. 
The expression $A:B$ refers to the scalar product of the matrices $A,B\in \R^{3\times 3}$, i.e. $A:B = \sum_{i,j} A_{ij} B_{ij}$.
To write this in an abstract setting we define $A=A_{\lambda, \mu} \in L^\infty(I; \Lsa(\spaceV; \spaceVdual))$ and $C=C_\rho\in L^\infty(I; \Lsa(\spaceH))$ through 
\begin{align*}
  \dup{A(t)v}\phi &= \int_\Omega \sigma_{\lambda(t), \mu(t)}(v): \epsilon(\phi) \dx = \int_\Omega 2\mu(t) \epsilon(v):\epsilon(\phi) + \lambda(t) \divv(v)\divv(\phi) \dx, \\
  C(t)u &= \rho(t) u 
\end{align*}
for $u\in \spaceH$ and $v,\phi\in \spaceV$. 
Therefore we are interested in finding a function $u\in L^2(I; \spaceV) \cap H^1(I; \spaceH)$ with $\mathcal Cu' \in H^1(I; \spaceVdual)$ which solves
\begin{subequations}\label{eq:elasticity:problem_abstract}
\begin{gather}
  (\mathcal Cu')'+ \mathcal A u = f \text{ in } L^2(I; \spaceVdual),\\
    u(0)=0\text{ in } \spaceH,\ (\mathcal C u')(0)=0\text{ in } \spaceVdual.
\end{gather}
\end{subequations}
This problem fits into the abstract theory of the previous sections, which yields results for the operator $S: (A, C)\mapsto u$. 
In particular, we conclude that $\map S{D(S)\subset X}Y$, where 
\begin{align*}
  X &= W^{1,\infty}(I; \Lsa(\spaceV;\spaceVdual)) \times W^{1,\infty}(I; \Lsa(\spaceH)), \\
  D(S) &= \left \{(A, C)\in X \ \Big| \ A\in L^\infty(I; \Lsa_{a_0+\epsilon}(\spaceV;\spaceVdual)) \text{ and } \right. \\
  &\qquad\qquad\qquad\quad\ \ \left. C\in  L^\infty(I; \Lsa_{c_0+\epsilon}(\spaceH)) \text{ for some $\epsilon > 0$}\right \},\\
  Y &= L^\infty(I; \spaceV)\cap W^{1,\infty}(I; \spaceH),
\end{align*}
is well-defined for $f\in L^2(I; \spaceH)$ or $f\in H^1(I; \spaceVdual)$. 
For a smooth $f \in \mathcal F^{(k)}$, we can also regard $S$ as a mapping $\map S{D(S)\cap X^{(k)}\subset X^{(k)}}{Y^{(k)}}$ with
\begin{align*}
  X^{(k)} &= W^{k+1,\infty}(I; \Lsa(\spaceV;\spaceVdual)) \times W^{k+1,\infty}(I; \Lsa(\spaceH))\\
  Y^{(k)} &= W^{k,\infty}(I; \spaceV)\cap W^{k+1,\infty}(I; \spaceH).
\end{align*}

We compose $S$ with the operator
\begin{equation*}
  P (\lambda, \mu, \rho) = (A_{\lambda, \mu}, C_\rho)
\end{equation*}
to get the forward operator $F=S\circ P$ of our problem. 
It is well-defined for those $\lambda, \mu, \rho$ that are mapped onto $D(S)\cap X$ by $P$. 
The following lemma characterizes this set of functions.

\begin{lemma}\label{lemma:elasticity:continuous_p}
  Let $k\geq 0$ and $\lambda, \mu, \rho \in W^{k+1,\infty}(I; L^\infty(\Omega))$. 
  Then we have $(A_{\lambda, \mu}, C_\rho)\in X^{(k)}$ with
  \begin{alignat*}{3}
    &\norm{A_{\lambda,\rho}}_{W^{k+1,\infty}(I; \mathcal L(\spaceV,\spaceVdual))} &&\leq 2 \norm{\mu}_{W^{k+1,\infty}(I; L^\infty(\Omega))}+\norm{\lambda}_{W^{k+1,\infty}(I; L^\infty(\Omega))} \\
    &\norm{C_\rho}_{W^{k+1,\infty}(I; \mathcal L(\spaceH))} &&\leq \norm{\rho}_{W^{k+1,\infty}(I; L^\infty(\Omega))}
  \end{alignat*}
  If, in addition, $\rho(t,x)\geq \rho_0$, $\mu(t,x) \leq \alpha_0$ and $\alpha_0^{-1} \leq 2\mu(t,x)+3\lambda(t,x)\leq \alpha_0$ for some $\rho_0, \alpha_0 > 0$, then 
  \begin{equation*}
    \dup{A(t)\phi}{\phi} \geq \alpha_0 \norm{\epsilon(\phi)}^2_{L^2(\Omega, \R^{3\times 3})}\ \ \text{and}\ \ \scp{C(t)\psi}{\psi}\geq \rho_0 \norm{\psi}^2_{L^2(\Omega, \R^3)}
  \end{equation*}
  for all $\phi\in \spaceV$, $\psi\in \spaceH$ and almost all $t\in I$.
\end{lemma}

\begin{proof}
  The norm estimates are straightforward, as is the coercivity of $C(t)$. 
  Regarding the coercivity of $A(t)$ see, e.g.,~\cite{kirsch_inverse_2016}.
\end{proof}

Let $\rho_0>0$ and $\alpha_0>0$ be fixed in the sequel, and let $C_K$ and $C_P$ denote the constants from the Korn- and Poincaré inequality for $\Omega$, respectively, i.e. 
$\norm{\epsilon(\phi)}^2_{L^2(\Omega, \R^3)} \geq C_K \norm{\nabla \phi}^2_{L^2(\Omega, \R^3)} \geq C_K C_P \norm{\phi}^2_{\spaceV}$ holds for all $\phi\in V$. 
According to the above lemma, $P$ given as
\begin{equation*}
  P:D(P)\cap W^{(k)}\subset W^{(k)}\to D(S)\cap X^{(k)}\subset X^{(k)}
\end{equation*}
is well-defined for $k\geq 0$ if we set the constants that appear in the definition of $D(S)$ to be 
$C_0 = \rho_0$ and $A_0 = \alpha_0 C_K C_P$ and define the spaces
\begin{align*}
  W^{(k)} &= \left(W^{k+1,\infty}(I; L^\infty(\Omega))\right)^3,  \\
  D(P) &= \Big \{(\lambda, \mu, \rho) \in W^{(0)} \,|\, \rho \geq \rho_0+\epsilon, \ \mu \leq \alpha_0 - \epsilon \text{ and }  \\
       &\ \qquad \alpha_0^{-1}+\epsilon \leq 2\mu(t,x)+3\lambda(t,x)\leq \alpha_0-\epsilon \text{ a.e.\ in $S\times \Omega$ for some $\epsilon>0$}\Big \}.
\end{align*}
The forward operator can therefore be considered as the mapping 
\begin{equation*}
  \map{F\assign S\circ P}{D(P)\cap W^{(k)}}{Y^{(k)}}
\end{equation*}
for arbitrary $k\geq0$. 
We note that $D(P)\cap W^{(k)}$ is an open subset of the Banach space $W^{(k)}$ (cf. \Cref{lemma:abstract:lsa_interior}), so analyzing the Fréchet-differentiability of $F$ (and $P$) makes sense. 

\subsection{Properties of the Forward Operator}

We already know about the differentiability of $S$, so we only need to discuss derivatives of $P$. 
In this setting $P$ is linear and continuous, so we can directly calculate $\partial F$ using the chain rule and \Cref{corollary:abstract:frechet_total}. 

\begin{theorem}\label{theorem:elasticity:frechet_f}
  Let $k\geq 2$ and $f\in \mathcal F^{(k)}$. 
  Then $F: D(P) \cap W^{(k)} \to Y^{(k-2)}$ is Fréchet-differentiable.
  For all $x=(\lambda, \mu, \rho)\in D(P)\cap W^{(k)}$ and $h=(\bar\lambda, \bar\mu, \bar\rho)\in W^{(k)}$, $\partial F(x)[h]$ is given as the unique weak solution $u_h$ of the equation
  \begin{align*}
      (\rho u_h')'(t) - \divv \left(2\mu(t) \epsilon(u_h(t))+ \lambda \divv u_h(t) I_3 \right) &= -A_{\bar\lambda, \bar\mu}(t) u(t) - (\mathcal C_{\bar\rho}u')'(t) \\
      &= \divv \left(2\bar\mu(t) \epsilon(u(t))+ \bar\lambda \divv u(t) I_3 \right) - (\bar\rho u')'(t)
  \end{align*}
  that also satisfies homogeneous initial values $u_h(0) = (\rho u_h')(0) = 0$. 
  As always, $u=F(x)$ denotes the solution of the forward problem.
\end{theorem}

We continue with the adjoint of $\partial F(x) \in \mathcal L(W^{(k)}, L^2(I; \spaceH))$. 
We already know about $\partial S(P(x))^*$, and due to the chain rule have
\begin{equation*}
  \partial F(x)^* =  P^* \circ \partial S(P(x))^*.
\end{equation*}
This formula suggests that we should also analyze $P^*$ independently from $S$. 
However, even with the simple structure of $C(t)$ a characterization of $P^*\in \mathcal L((X^{(k)})^*, (W^{(k)})^*)$ is not possible because of insufficent knowledge about the dual space of $X^{(k)}$, i.e.\ how a general $v\in (X^{(k)})^*$ could act on $P(h)$. 
Fortunately we do not need to evaluate $P^*(z)$ for arbitrary $z$, but only for $z\in \mathcal R(\partial S(P(x))^*)$. 
From \Cref{theorem:abstract:adjoints} we know that these $z\in (X^{(k)})^*$ evaluate its argument at a point (depending on $x$) and form a kind of $L^2(I; \spaceH)$ inner product with the result.

Since the abstract formulation of our elastic wave equation has $B=Q=0$, the adjoint equation~\eqref{eq:abstract:adjoint_equation} in this case is the original equation that has to be solved backwards in time. 

\begin{theorem}
  Let $k\geq 2$, $f\in \mathcal F^{(k)}$ and $x\in D(P)\cap W^{(k)}$.
  The application of the adjoint of $\partial F(x)\in \mathcal L(W^{(k)}, L^2(I; \spaceH))$ on $v\in L^2(I; \spaceH)$ can be written as 
  \begin{equation*}
    \partial F(x)^*[v] = \begin{pmatrix*}[l]
      -\divv u\divv w_v \\
      -2\epsilon(u):\epsilon(w_v) \\
      \displaystyle u'\cdot w_v'  
     \end{pmatrix*} 
     \in L^1(I; L^1(\Omega))^3 \subset \left( L^\infty(I; L^\infty(\Omega))^3 \right)^* \subset \left(W^{(k)}\right)^*,
  \end{equation*}
  where the embedding of $L^1(I; L^1(\Omega))$ into $L^\infty(I; L^\infty(\Omega))^*$ has to be understood using the inner product of $L^2(I; L^2(\Omega))$, $u=F(x)\in Y^{(2)}$ and $w_v\in Y$ denotes the solution of
  \begin{align*}
    (\rho w_v')'(t) - \divv \left(2\mu(t) \epsilon(w_v(t))+ \lambda \divv w_v(t) I_3 \right) = v(t) \text{\ in $\spaceVdual$}
  \end{align*}
  for almost all $t\in I$ together with homogeneous end conditions $(\rho w_v')(T)=w_v(T)=0$. 
\end{theorem}

\begin{proof}
  Let $h=(\bar\lambda, \bar\mu, \bar\rho)\in W^{(k)}$. 
  We use the characterization of $\partial S(P(x))^*$ from \Cref{corollary:abstract:adjoint_total} to see
  \begin{align*}
    \dup{\partial F(x)^*[v]}h_{(W^{(k)})^*\times W^{(k)}} &= \dup{\partial S(P(x))^*[v]}{P(h)}_{(X^{(k)})^*\times X^{(k)}} \\
    &= \int_0^T \scp{C_{\bar\rho}u'(t)}{w_v'(t)} - \dup{A_{\bar\lambda,\bar\mu}(t)u(t)}{w_v(t)} \dt  \\
    &= \int_0^T \scp{\bar\rho(t)u'(t)}{w_v'(t)} - \scp{\bar\lambda(t) \divv u(t)}{\divv w_v(t)}  \\
    &\qquad - \int_\Omega 2 \bar\mu(t) \epsilon(u(t)):\epsilon(w_v(t)) \dx \dt.
  \end{align*}
  Since e.g. $\bar\rho(t)\in L^\infty(\Omega)$ and $u'(t)\cdot w_v'(t)\in L^1(\Omega)$ we can also write this in a way that the integrands are dual products of the linearization parameters, i.e. 
  \begin{align*}
    \dup{\partial F(x)^*[v]}h_{(W^{(k)})^*\times W^{(k)}} &= \int_0^T \dup{u'(t)\cdot w_v'(t)}{\bar\rho(t)} - \dup{\divv u(t)\divv w_v(t)}{\bar\lambda(t)} \\
    &\qquad - \dup{2\epsilon(u(t)):\epsilon(w_v(t))}{\bar\mu(t)} \dt.
  \end{align*}
  Here, $\dup{\cdot}{\cdot}$ denotes the dual product between $L^1(\Omega) \subset L^\infty(\Omega)^*$ and $L^\infty(\Omega)$. 
  This can also be done with the time variable, which then proves the assertion.
\end{proof}

\subsection{Ill-posedness}

In \Cref{theorem:abstract:illposed_locally} we showed the ill-posedness of $S$ by constructing suitable sequences of arguments. 
These sequences do not lie in the range of $P$, so we cannot directly use that result to conclude ill-posedness of $F$. 
Instead, we construct sequences of parameters such that their image under $P$ fulfills the assumptions of \Cref{theorem:abstract:illposed_convergence}.

In the abstract setting we used time-independent disturbances $R_j$. 
This time we decide to make them independent of the spatial variables instead.
For this we need the following lemma: 

\begin{lemma}\label{lemma:elasticity:illposed_helper}
  Let $r\in\N\cup \{0\}$. 
  There exists $(\alpha_j)_{j\in\N}\subset C_c^\infty(I)$ which satisfies
  \begin{equation*}
    0 < \gamma \leq \norm{\alpha_j}_{W^{r,\infty}(I)} \leq 1 \quad \text{for all $j\in \N$}
  \end{equation*}
  and $\alpha_j\cdot \phi \to 0$ in $H^m(I)$ when $j\to\infty$ for all fixed $\phi \in H^m(I)$ with $m=0,\dots, r$.
\end{lemma}

\begin{proof}
  Let $t_0\in (0,T)$ and $\psi\in C_c^\infty(\R)$ such that $\spt \psi = [-1,1]$, $\psi(t)\in [0,1]$ for all $t\in\R$ and $\norm{\psi}_{W^{r,\infty}(I)} = 1$. 
  We define
  \begin{equation*}
    \alpha_j(t) = j^{-r} \psi(j(t-t_0)).
  \end{equation*}
  If $j> \max \{\sfrac 1{t_0}, \sfrac 1{(T-t_0)}\}$ then $\alpha_j\in C_c^\infty(I)$. 
  Hence we might have to drop elements from the beginning of this sequence, but without loss of generality we assume that this is not the case.
  We see that $\spt \alpha_j = t_0 + [-\sfrac 1j , \sfrac 1j]$ and $\alpha_j^{(i)}(t) = j^{i-r} \psi^{(i)}(j(t-t_0))$, so
  \begin{equation*}
    \norm{\alpha_j}_{W^{r,\infty}(I)} = \max_{i=0,\dots, r} j^{i-r} \norm{\psi^{(i)}}_{L^\infty(S)}
    \ \ \begin{cases}
      \leq 1 \\
      \geq \norm{\psi^{(r)}}
    \end{cases}
  \end{equation*}
  holds. 
  For arbitrary $\phi \in H^m(I)$, $m\in \{0, \dots, r\}$ we have 
  \begin{align*}
    \norm{\alpha_j \phi}^2_{H^m(I)} &= \sum_{i=0}^m \sum_{l=0}^i {i \choose l} \int_0^T \left|\alpha_j^{(i-l)}(t) \phi^{(l)}(t) \right|^2 \dt \\
    &\leq \sum_{i=0}^m \sum_{l=0}^i {i \choose l} \int_{t_0-\sfrac 1j}^{t_0+\sfrac 1j} \left|\phi^{(l)}(t) \right|^2 \dt \to 0,
  \end{align*}
  because due to $\phi^{(l)}\in L^2(I)$ for $l=0,\dots, m$ the function $|\phi^{(l)}(t)|^2$ is integrable and dominates the integrand $\chi_{[t_0-\sfrac 1j,t_0+\sfrac 1j]}(t) |\phi^{(l)}(t)|^2$, which converges pointwise to zero. 
\end{proof}

Now we use these functions to construct suitable sequences of parameters.

\begin{theorem}\label{theorem:elasticity:illposed_locally}
  Let $k\in \N$ and $f\in \mathcal F^{(k)}$.
  Then the task of finding $\lambda$, $\mu$ or $\rho$ such that $F(\lambda, \mu, \rho)=y\in Y^{(k-1)}$ is locally ill-posed in every $(\lambda, \mu, \rho)\in D(P) \cap W^{(k)}$.
\end{theorem}

\begin{proof}
  Let $p=(\lambda, \mu, \rho)\in D(P) \cap W^{(k)}$ and $P(p) = (A, C)\in X^{(k)}$.
  Since $D(P)\cap W^{(k)}$ is an open subset of $W^{(k)}$ there exists $\delta_0>0$ with $B(p,\delta_0)\subset D(P)\cap W^{(k)}$. 
  Let $0<\delta\leq \delta_0$ be fixed and $(\alpha_j)_{j\in\N}$ be the sequence from \Cref{lemma:elasticity:illposed_helper} with $r=k+1$.
  \begin{description}
    \item[\normalfont{Identification of $\rho$:}] We set $\rho_j(t,x) = \rho(t,x) + \delta \alpha_j(t)/2$ and $p_j\assign (\lambda, \mu, \rho_j)$.
    This way $\rho_j\in B(\rho, \delta)$ but $\rho_j\not\to \rho$, both in the norm of $W^{k+1,\infty}(I; L^\infty(\Omega))$.
    Note that $P(p_j) = (A, C+R_j^C)$ with
    \begin{equation*}
      R_j^C(t)\phi = \frac{\delta \alpha_j(t)}2\ \phi \in \spaceH
    \end{equation*}
    for all $\phi\in\spaceH$ and almost all $t\in I$. 
    The norm of $R_j^C$ stays bounded for $j\to\infty$ due to continuity of $P$. 
    Moreover, for $u\in H^{k+1}(I; \spaceH)$ we see that
    \begin{equation*}
      \norm{(\mathcal R_j^C u')'}_{H^{k-1}(I; \spaceH)} \leq \norm{R_j^C u'}_{H^{k}(I; \spaceH)} = \frac\delta2 \norm{\alpha_j(\cdot) \norm{u'(\cdot)}_{\spaceH}}_{H^k(I)} \to 0
    \end{equation*}
    for $j\to\infty$ since $\norm{u'(\cdot)} \in H^k(I)$.
    Therefore we can apply \Cref{theorem:abstract:illposed_convergence} to conclude that $F(p_j) = S(A, C+R_j^C)\to S(A, C) = F(p)$ in $Y^{(k-1)}$ when $j\to\infty$.
    \item[\normalfont{Identification of $\lambda$ or $\mu$:}] Continuing in the same fashion, we set $\lambda_j(t,x) = \lambda(t,x) + \delta \alpha_j(t)/2$ and $p_j\assign (\lambda_j, \mu, \rho)$. 
    Hence, $P(p_j) = (A+R_j^A, C)$ with
    \begin{equation*}
      R_j^A(t)\phi = - \frac{\delta}2 \alpha_j(t) \divv(\divv \phi I_3) \in \spaceVdual
    \end{equation*}
    for all $\phi\in\spaceV$ and almost all $t\in I$. 
    Again, the norm of $R_j^A$ stays bounded for $j\to\infty$.
    For $v\in H^k(I; \spaceV)$ we see that
    \begin{equation*}
      \norm{\mathcal R_j^A v}_{H^k(I; \spaceVdual)} \leq \frac\delta2 \norm{\alpha_j(\cdot) \norm{v(\cdot)}_{\spaceV}}_{H^k(I)} \to 0
    \end{equation*}
    for $j\to\infty$. 
    This enables us to apply \Cref{theorem:abstract:illposed_convergence} once again.
    For $\mu$ we can do the same with $R_j^A(t)\phi = - \delta\alpha_j(t) \divv \epsilon(\phi)\in\spaceVdual$. \qedhere
  \end{description}
\end{proof}

When applying a Newton solver to the nonlinear inverse problem, it is even more important to know whether the linearization of $F$ is ill-posed. 

\begin{corollary}\label{corollary:elasticity:frechet_compact}
  Let $k\geq 2$ and $f\in \mathcal F^{(k)}$.
  We consider $\map F{D(P) \cap W^{(k)}}{Z}$ with $Z=W^{j,p}(I; H)$ or $Z=C^j(I; H)$ for $0\leq j \leq k$ and $1\leq p< \infty$.
  For every  $x=(\lambda, \mu, \rho)\in D(P)\cap W^{(k)}$ its linearization $\partial F(x)\in \mathcal L(W^{(k)}, Z)$ is a compact operator.
\end{corollary}

\begin{proof}
  $\partial F(x) = \partial S(P(x)) \circ P$ with linear and continuous $P$ and compact $\partial S(P(x))$ (because of \Cref{lemma:abstract:frechet_compact}).
\end{proof}

It could be that $\partial F(x)$ is only compact because it has finite dimensional range, which would make the resulting problems well-posed in the sense of linear inverse problems (ill-posed in the sense of Hadamard, but with a continuous generalized inverse). 
This is not the case.

\begin{lemma}\label{lemma:elasticity:frechet_range}
  Let $k\geq 2$ and $f\in \mathcal F^{(k)}$.
  For all $x\in D(P)\cap W^{(k)}$ the ranges of 
  \begin{equation*}
    \partial_\lambda F(x), \partial_\mu F(x), \partial_\rho F(x) \in \mathcal L(W^{k+1,\infty}(I; L^\infty(\Omega)), Y^{(k-1)})
  \end{equation*}
  are of infinite dimension. 
\end{lemma}

\begin{proof}
  The argument is very similar to the one used for \Cref{lemma:abstract:frechet_range}, but we have to verify that the operators from the corresponding proof can be reached by $P$.
  When considering the identification of $\rho$ this is the case. 
  Denoting with $(\alpha_i)_{i\in\N}$ and $(h_i)_{i\in\N}$ the sequences from the proof of \Cref{lemma:abstract:frechet_range}, the choice $\bar\rho_i(t,x) = \alpha_i(t)$ yields $C_{\bar\rho_i}=h_i$ and therefore the assertion holds in this case.
  
  If we use $\bar\lambda_i(t,x) = \bar\mu_i(t,x)= \alpha_i(t)$, then the right-hand side of the linearized PDE w.r.t.\ $\lambda$ reads 
  $\alpha_i(t) \divv (\divv u(t) I_3)$, which (by the construction of the $\alpha_i$) yields a set of linearly independent functions if and only if $\divv (\divv u(t) I_3) \neq 0$. 
  For $\mu$ this right-hand side is $2 \alpha_i(t) \divv (\eps(u(t)))$. 
  Either $\divv (\divv u(t) I_3)=0$ or $2\divv (\eps(u(t)))=0$ would imply $u(t)=0$ (test with $u(t)$ and use coercivity of $A_{1, 0}$ and $A_{0,1}$), but the $\alpha_i$ were constructed in such a way that $u(t)\neq 0$ for all $t\in \spt \alpha_i$. 
\end{proof}

\section{Application to Electrodynamics}%
\label{section:maxwell}

As a second possible application of the abstract theory we choose a simple model based on Maxwell's equations, a second-order equation for the electrical field $E$ inside a bounded domain $\Omega\subset \R^3$, which reads
\begin{equation}\label{eq:maxwell:pde}
  (\epsilon E')' + \curl({\mu^{-1} \curl E}) = f.
\end{equation}
The equation is furnished with the initial- and boundary conditions $E(0)=E'(0)=0$, $E(t) = 0$ on $\partial\Omega$.

The goal is the analysis of the identification of a time- and space-dependent permittivity $\eps$ and permeability $\mu$. 
The treatment of this problem is very similar to the elastic wave equation of the last section, therefore we give a less detailed discussion of this problem.

\renewcommand{\spaceV}{\ensuremath{V}}
\renewcommand{\spaceVdual}{\ensuremath{V^*}}
\renewcommand{\spaceH}{\ensuremath{L^2(\Omega, \R^3)}}

Appropriate function spaces for equation~\eqref{eq:maxwell:pde} are 
\begin{equation*}
  V=\Set{E\in H^1_0(\Omega, \R^3) | \divv E = 0}\quad \text{and} \quad H=L^2(\Omega, \R^3).
\end{equation*}
For a first-order Maxwell system one would typically use the space $H_0(\curl, \Omega)$. 
Since we need $\curl^2$ to be coercive on $V$ we have to make the additional assumptions $\divv E = 0$ and that not only the tangential component, but also the normal component of $E$ vanishes on $\partial \Omega$. 
For smooth or convex $\Omega$ the set of $H_0(\curl, \Omega)$ functions that fulfill these restrictions coincides with $V$ (cf.~\cite{monk_finite_2008}).
We endow $V$ with the $H^1(\Omega, \R^3)$-norm and for notational purposes continue to abbreviate it using $V$.

The weak formulation can again be written in the form $(\mathcal C u')'+\mathcal Au = f$ with $u=E$ and the operators $C=C_\eps$, $A=A_\mu$ are given as
\begin{equation*}
  C(t)v = \epsilon(t) v \quad \text{and}\quad \dup{A(t)\psi}{\phi} = \scp{\mu(t)^{-1} \curl \psi}{\curl \phi}_{\spaceH}
\end{equation*}
for $v\in \spaceH$ and $\phi, \psi\in \spaceV$.

The operator that maps $A$ and $C$ onto $E$ is almost the same one as the one that was used in the elastic setting, only the function spaces are different: $\map S{D(S)\subset X}Y$, where 
\begin{align*}
  X &= W^{1,\infty}(I; \Lsa(\spaceV;\spaceVdual)) \times W^{1,\infty}(I; \Lsa(\spaceH)), \\
  D(S) &= \left \{(A, C)\in X \ \Big| \ A\in L^\infty(I; \Lsa_{a_0+\delta}(\spaceV;\spaceVdual)) \text{ and } \right. \\
  &\qquad\qquad\qquad\quad\ \ \left. C\in  L^\infty(I; \Lsa_{c_0+\delta}(\spaceH)) \text{ for some $\delta > 0$}\right \},\\
  Y &= L^\infty(I; \spaceV)\cap W^{1,\infty}(I; \spaceH),
\end{align*}
is well-defined for $f\in L^2(I; \spaceH)$ or $f\in H^1(I; \spaceVdual)$. 

For $k\geq 1$ and $f\in \mathcal F^{(k)}$ we can define $\map S{D(S)\cap X^{(k)}\subset X^{(k)}}{Y^{(k)}}$ using
\begin{align*}
  X^{(k)} &= W^{k+1,\infty}(I; \Lsa(V,V^*)) \times W^{k+1,\infty}(I; \Lsa(\spaceH))\\
  Y^{(k)} &= W^{k,\infty}(I; V)\cap W^{k+1,\infty}(I; \spaceH).
\end{align*}

We compose $S$ with the operator
\begin{equation*}
  P (\eps, \mu) = (A_{\mu}, C_\eps)
\end{equation*}
to get the forward operator $F=S\circ P$ of our problem. 
The mapping properties of $P$ are more complicated to derive because we need to estimate derivatives of $\mu^{-1}$ w.r.t.\ time up to order $k+1$. 
The following (easy to prove) formula takes care of this.

\begin{lemma}\label{lemma:maxwell:reciprocal_deriv}
  Let $n\geq 1$ und $\mu\in W^{m,\infty}(I)$ with $z(t)\geq z_0 > 0$ almost everywhere. 
  Then $1/z(\cdot)$ is also $m$-times weakly differentiable and satisfies
  %
  %
  %
  \begin{equation*}
        \norm{\frac 1z}_{W^{m,\infty}(I)} \leq C(m)
        \left(1+\frac{1}{z_0}\right)^{m+1} \left(1+\norms{z}_{W^{m,\infty}(I)}\right)^{m}.
  \end{equation*}
\end{lemma}

Now we can state the analog of \Cref{lemma:elasticity:continuous_p} for our Maxwell-model. 

\begin{lemma}\label{lemma:maxwell:continuous_p}
  Let $k\geq 0$ and $\eps, \mu \in W^{k+1,\infty}(I; L^\infty(\Omega))$ with $\eps(t,x) \geq \eps_0$ and $\mu_1 \geq \mu(t,x) \geq \mu_0$ for some $\mu_1, \mu_0, \eps_0>0$ then we have $(A_{\mu}, C_\eps)\in X^{(k)}$ with
  \begin{alignat*}{3}
    &\norm{A_{\mu}}_{W^{k+1,\infty}(I; \mathcal L(\spaceV,\spaceVdual))} &&\leq C(k) \left(1+\frac{1}{\mu_0}\right)^{k+2} \left(1+\norms{\mu}_{W^{k+1,\infty}(I; L^\infty(\Omega))}\right)^{k+1} \\
    &\norm{C_\eps}_{W^{k+1,\infty}(I; \mathcal L(\spaceH))} &&\leq \norm{\eps}_{W^{k+1,\infty}(I; L^\infty(\Omega))}
  \end{alignat*}
  and that 
  \begin{equation*}
    \dup{A(t)\phi}{\phi} \geq \mu_1 \norm{\curl \phi}^2_{L^2(\Omega, \R^{3})}\ \ \text{and}\ \ \scp{C(t)\psi}{\psi}\geq \eps_0 \norm{\psi}^2_{L^2(\Omega, \R^3)}
  \end{equation*}
  holds for all $\phi\in \spaceV$, $\psi\in \spaceH$ and almost all $t\in I$.
\end{lemma}

\begin{proof}
 The coercivity of both operators is very easy to see, as is the norm estimate for $C$. 
 The norm estimate for $A$ follows from \Cref{lemma:maxwell:reciprocal_deriv}.
\end{proof}

Let the constants $\mu_0,\mu_1$ and $\eps_0$ be fixed in the sequel.
Since $V$ contains those functions from $H^1_0(\Omega,\R^3)$ that are divergence free, $\norm{\curl \cdot}_{L^2}$ is equivalent to the $H^1_0(\Omega)$-norm. 
For smooth $\phi\in C_0^\infty(\Omega, \R^3)$ with $\divv \phi = 0$ we see that 
\begin{equation*}
 \int_\Omega |\curl \phi|^2 \dx = \int_\Omega \phi (\curl^2 \phi - \nabla (\divv \phi)) \dx = -\int_\Omega \Delta \phi \dx = \int_\Omega |\nabla \phi|^2 \dx.
\end{equation*}
By approximation, $\norm{\curl \phi}_{L^2}^2 = \norm{\nabla\phi}_{L^2}^2 \geq C_P \norm{\phi}^2_{V}$ holds for all $\phi\in V$. 
Here, $C_P$ denotes the Poincaré-constant of $\Omega$.
According to the above lemma and these considerations, $P$ given as
\begin{equation*}
  P:D(P)\cap W^{(k)}\subset W^{(k)}\to D(S)\cap X^{(k)}\subset X^{(k)}
\end{equation*}
is well-defined for $k\geq 0$ if we set the constants that appear in the definition of $D(S)$ to be 
$C_0 = \eps_0$ and $A_0 = \mu_1 C_P$ and introduce the spaces
\begin{align*}
  W^{(k)} &= \left(W^{k+1,\infty}(I; L^\infty(\Omega))\right)^2,  \\
  D(P) &= \Big \{(\eps, \mu) \in W^{(0)} \,|\, \mu_1-\delta \geq \mu \geq \mu_0+\delta, \ \eps \geq \eps_0 + \delta   \\
       &\ \qquad \text{ a.e.\ in $S\times \Omega$ for some $\delta>0$}\Big \}.
\end{align*}
The forward operator can therefore be considered as the mapping 
\begin{equation*}
  \map{F\assign S\circ P}{D(P)\cap W^{(k)}}{Y^{(k)}}
\end{equation*}
for arbitrary $k\geq0$. 

\subsection{Properties of the Forward Operator}

In contrast to the elastic equation we now have to deal with a nonlinear $P$, but its derivative is easy to calculate:

\begin{lemma}\label{theorem:maxwell:frechet_p}
  For every $k\geq 0$ the operator $P:D(P)\cap W^{(k)}\to X^{(k)}$ is Fréchet-differentiable and its derivative $\partial P:D(P)\cap W^{(k)}\to \mathcal L(W^{(k)},X^{(k)})$ is given for all $(\eps,\mu)\in D(P)\cap W^{(k)}, (\bar\eps, \bar\mu)\in W^{(k)}$ by
  \begin{align*}
    \partial P(\eps, \mu)[\bar\eps, \bar\mu] &= \begin{pmatrix*}[l]
      \partial A_\mu[\bar\mu] \\
      C_{\bar\eps}
    \end{pmatrix*} = t\mapsto
    \begin{pmatrix*}[l]
      u \in \spaceV \mapsto -\curl \left(\frac{\bar\mu(t)}{\mu(t)^2} \curl u\right) \in \spaceVdual \\
      u \in \spaceH \mapsto \bar\eps(t) u\in \spaceH
    \end{pmatrix*}.
  \end{align*}
\end{lemma}

An application of the chain rule yields the following theorem for the derivative of $F$.

\begin{theorem}\label{theorem:maxwell:frechet_f}
  Let $k\geq 2$ and $f\in \mathcal F^{(k)}$. 
  Then $F: D(P) \cap W^{(k)} \to Y^{(k-2)}$ is Fréchet-differentiable.
  For all $x=(\eps, \mu)\in D(P)\cap W^{(k)}$ and $h=(\bar\eps, \bar\mu)\in W^{(k)}$, $\partial F(x)[h]$ is given as the unique weak solution $E_h$ of the equation
  \begin{align*}
      (\epsilon E_h')'(t) + \curl (\mu(t)^{-1} \curl E_h(t)) &= -\partial A_{\mu}[\bar \mu](t) E(t) - (\mathcal C_{\bar\eps}E')'(t) \\
      &= \curl \left(\frac{\bar\mu(t)}{\mu(t)^2} \curl E(t)\right) - (\bar\eps E')'(t)
  \end{align*}
  that also satisfies homogeneous initial values $E_h(0) = (\eps E_h')(0) = 0$. 
  With $E=F(x)$ we denote the solution of the forward problem.
\end{theorem}

We continue with the adjoint of $\partial F(x) \in \mathcal L(W^{(k)}, L^2(I; \spaceH))$. 
The characterization is obtained in the same way as for the elastic case.

\begin{theorem}\label{theorem:maxwell:adjoint}
  Let $k\geq 2$, $f\in \mathcal F^{(k)}$ and $x\in D(P)\cap W^{(k)}$.
  The application of the adjoint of $\partial F(x)\in \mathcal L(W^{(k)}, L^2(I; \spaceH))$ on $v\in L^2(I; \spaceH)$ can be written as 
  \begin{equation*}
    \partial F(x)^*[v] = \begin{pmatrix*}[l]
      -\mu^{-2}\curl u\cdot \curl w_v \\
      \displaystyle u'\cdot w_v'  
     \end{pmatrix*} 
     \in L^1(I; L^1(\Omega))^2 \subset \left( L^\infty(I; L^\infty(\Omega))^2 \right)^* \subset \left(W^{(k)}\right)^*,
  \end{equation*}
  where $u=F(x)\in Y^{(2)}$ and $w_v\in Y$ denotes the solution of
  \begin{align*}
    (\epsilon w_v')'(t) + \curl (\mu(t)^{-1} \curl w_v(t)) = v(t) \text{\ in $\spaceVdual$}
  \end{align*}
  for almost all $t\in I$ together with homogeneous end conditions $(\rho w_v')(T)=w_v(T)=0$. 
\end{theorem}

\begin{proof}
  Let $h=(\bar\eps, \bar\mu)\in W^{(k)}$. 
  We use the characterization of $\partial S(P(x))^*$ from \Cref{corollary:abstract:adjoint_total} to see
  \begin{align*}
    \dup{\partial F(x)^*[v]}h_{(W^{(k)})^*\times W^{(k)}} &= \dup{\partial S(P(x))^*[v]}{P(h)}_{(X^{(k)})^*\times X^{(k)}} \\
    &= \int_0^T \scp{C_{\bar\eps}E'(t)}{w_v'(t)} - \dup{\partial A_{\mu}[\bar\mu](t)E(t)}{w_v(t)} \dt  \\
    &= \int_0^T \scp{\bar\eps(t)E'(t)}{w_v'(t)} + \scp{\bar\mu(t) \mu(t)^{-2} \curl E(t)}{\curl w_v(t)} \dt.
  \end{align*}
  Since e.g. $\bar\eps(t)\in L^\infty(\Omega)$ and $u'(t)\cdot w_v'(t)\in L^1(\Omega)$ we can also write this in a way that the integrands are dual products with the linearization parameters on one side, i.e. 
  \begin{align*}
    \dup{\partial F(x)^*[v]}h_{(W^{(k)})^*\times W^{(k)}} &= \int_0^T \dup{E'(t)\cdot w_v'(t)}{\bar\eps(t)} + \dup{\mu(t)^{-2}\curl E(t)\cdot \curl w_v(t)}{\bar\mu(t)} \dt.
  \end{align*}
  Here, $\dup{\cdot}{\cdot}$ denotes the dual product between $L^1(\Omega) \subset L^\infty(\Omega)^*$ and $L^\infty(\Omega)$. 
  This can also be done with the time variable, which then proves the assertion.
\end{proof}

\subsection{Ill-posedness}

\begin{theorem}\label{theorem:maxwell:illposed_locally}
  Let $k\in \N$ and $f\in \mathcal F^{(k)}$.
  Then the task of finding $\eps$ or $\mu$ such that $F(\eps, \mu)=y\in Y^{(k-1)}$ is locally ill-posed in every $(\eps, \mu)\in D(P) \cap W^{(k)}$.
\end{theorem}

\begin{proof}
  Let $p=(\eps, \mu)\in D(P) \cap W^{(k)}$ and $P(p) = (A, C)\in X^{(k)}$.
  Since $D(P)\cap W^{(k)}$ is an open subset of $W^{(k)}$ there exists $\delta_0>0$ with $B(p,\delta_0)\subset D(P)\cap W^{(k)}$. 
  Let $0<\delta\leq \delta_0$ be fixed and $(\alpha_j)_{j\in\N}$ be the sequence from \Cref{lemma:elasticity:illposed_helper} with $r=k+1$.
  The ill-posedness of the reconstruction of $\eps$ can be done using the exact same proof as for $\rho$ in the elastic case. 
  Regarding $\mu$: We define $\mu_j(t,x) = (\mu(t,x)^{-1} + \delta \alpha_j(t)/2)^{-1}$ and $p_j\assign (\eps, \mu_j)$. 
  Hence, $P(p_j) = (A+R_j^A, C)$ with
    \begin{equation*}
      R_j^A(t)\phi = \frac{\delta}2 \alpha_j(t) \curl^2 \phi \in \spaceVdual
    \end{equation*}
    for all $\phi\in\spaceV$ and almost all $t\in I$. 
    Again, the norm of $R_j^A$ stays bounded for $j\to\infty$.
    For $v\in H^k(I; \spaceV)$ we see that
    \begin{equation*}
      \norm{\mathcal R_j^A v}_{H^k(I; \spaceVdual)} \leq \frac\delta2 \norm{\alpha_j(\cdot) \norm{v(\cdot)}_{\spaceV}}_{H^k(I)} \to 0
    \end{equation*}
    for $j\to\infty$. 
    This enables us to apply \Cref{theorem:abstract:illposed_convergence} once again.
\end{proof}

\begin{corollary}\label{corollary:maxwell:frechet_compact}
  Let $k\geq 2$ and $f\in \mathcal F^{(k)}$.
  We consider $\map F{D(P) \cap W^{(k)}}{Z}$ with $Z=W^{j,p}(I; H)$ or $Z=C^j(I; H)$ for $0\leq j \leq k$ and $1\leq p< \infty$.
  For every  $x=(\lambda, \mu, \rho)\in D(P)\cap W^{(k)}$ its linearization $\partial F(x)\in \mathcal L(W^{(k)}, Z)$ is a compact operator.
\end{corollary}

\begin{proof}
  $\partial F(x) = \partial S(P(x)) \circ \partial P(x)$ with linear and continuous $\partial P(x)$ and compact $\partial S(P(x))$ (cf. \Cref{lemma:abstract:frechet_compact}).
\end{proof}

\begin{lemma}\label{lemma:maxwell:frechet_range}
  Let $k\geq 2$, $f\in \mathcal F^{(k)}$ and $f\neq 0$.
  For all $x\in D(P)\cap W^{(k)}$ the ranges of 
  \begin{equation*}
    \partial_\eps F(x), \partial_\mu F(x) \in \mathcal L(W^{k+1,\infty}(I; L^\infty(\Omega)), Y^{(k-1)})
  \end{equation*}
  are of infinite dimension. 
\end{lemma}

\begin{proof}
  The proof for $\eps$ was already done in the elastic case.   
  Denoting with $(\alpha_i)_{i\in\N}$ and $(h_i)_{i\in\N}$ the sequences constructed in the proof of \Cref{lemma:abstract:frechet_range}. 
  Choosing $\bar\mu_i(t,x) = \alpha_i(t)$, the right-hand side of the linearized PDE w.r.t. $\mu$ reads 
  $\alpha_i(t) \curl (\mu(t)^{-2} \curl E(t) )$, which (by the construction of the $\alpha_i$) yields a set of linearly independent functions if and only if $\curl (\mu(t)^{-2} \curl E(t)) \neq 0$. 
  The contrary would imply 
  \begin{equation*}
    0 = \int_\Omega \mu(t)^{-2} |\curl E(t)|^2 \dx \geq \mu_1^{-2} \int_\Omega |\curl E(t)|^2 \dx 
  \end{equation*}
  and we would conclude $E(t)=0$. 
  This cannot be the case for those $t$ where $\alpha_i(t)\neq 0$ because the $\alpha_i$ were constructed in such a way that $E(t)\neq 0$ for all $t\in \spt \alpha_i$. 
\end{proof}

\subsubsection*{Acknowledgements}

The author acknowledges funding by the Deutsche Forschungsgemeinschaft (DFG, German Research Foundation) – Project number 281474342/GRK2224/1.

\addcontentsline{toc}{section}{References}
\printbibliography%

\end{document}